\documentclass[twoside, 12pt]{article}
\usepackage{amsmath, amsthm, amssymb, amscd, mathptmx}
\usepackage{amsfonts}

\usepackage{color}
\usepackage[all]{xy}\CompileMatrices\SelectTips{cm}{12}

\theoremstyle{plain}
\newtheorem{Thm}{\sc Theorem}[section]
\newtheorem{Theorem}[Thm]{\sc Theorem}
\newtheorem{Corollary}[Thm]{\sc Corollary}
\newtheorem*{Corollary*}{\sc Corollary}
\newtheorem{Proposition}[Thm]{\sc Proposition}
\newtheorem*{Proposition*}{\sc Proposition}
\newtheorem{Lemma}[Thm]{\sc Lemma}

\theoremstyle{definition}
\newtheorem{Definition}[Thm]{Definition}

\theoremstyle{remark}
\newtheorem{Remark}[Thm]{Remark}
\newtheorem{Remarks}[Thm]{Remarks}

\newtheorem*{Example*}{Example}
\newtheorem*{Remark*}{Remark}

\newcommand{\EE}{{\mathbb E}}

\newcommand{\cS}{{\cal S}}

\newcommand{\id}{{\mathop{\rm id}}}
\newcommand{\ZZ}{{\mathbb Z}}
\newcommand{\FF}{{\mathbb F}}
\newcommand{\Flag}{\mathop{\rm Flag}}
\newcommand{\Spec}{\mathop{\rm Spec}}
\newcommand{\gr}{\mathop{{\rm gr}}}

\renewcommand{\O}{{\cal O}}
\newcommand{\sO}{{\mathcal O}}
\newcommand{\A}{{\cal A}}

\newcommand{\U}{{\cal U}}

\newcommand{\NN}{{\mathbb N}}

\newcommand{\PP}{{\mathbb P}}
\newcommand{\QQ}{{\mathbb Q}}

\newcommand{\CC}{{\mathbb C}}

\newcommand{\M}{{\cal M}}

\newcommand{\ti}{\tilde}
\newcommand{\tors}{\mathop{{\rm tors}}}
\newcommand{\Hom}{{\mathop{{\rm Hom}}}}
\newcommand{\cHom}{{\mathop{{\cal H}om}}}

\newcommand{\GL}{\mathop{\rm GL}}

\newcommand{\Pic}{{\mathop{\rm Pic\, }}}

\mathchardef\mhyp="2D
\newcommand{\LL}{{\mathbb L}}

\newcommand{\ml}[2]{\begin{multline}\label{#1}#2 \end{multline}}
\newcommand{\ga}[2]{\begin{gather}\label{#1}#2 \end{gather}}

\begin{document}

\markboth {\rm Monodromy}{On $p$-curvature conjecture}

\title{On a positive equicharacteristic variant of the
$p$-curvature conjecture}
\author{H\'el\`ene Esnault, Adrian Langer}
\date{April 15, 2012}

\maketitle


\begin{abstract}
Our aim is to formulate and prove a weak form in equal characteristic $p>0$ of
the $p$-curvature conjecture.   We also show the existence of a counterexample
to a strong form
of it.
\end{abstract}

\let\thefootnote\relax\footnote{The first author is supported by  the SFB/TR45
and the ERC Advanced Grant 226257. The second author is supported by the Bessel
Award of the Humboldt Foundation and a Polish MNiSW grant
(contract number N N201 420639). }

\section*{Introduction}
If $(E,\nabla)$ is a vector bundle with an algebraic integrable
connection over a smooth complex variety $X$, then it is defined
over a smooth scheme $S$ over $\Spec \ZZ[\frac{1}{N}]$ for some
positive integer $N$,  so $(E,\nabla)=(E_S, \nabla_S)\otimes_S
\CC$ over $X=X_S\otimes_S \CC$ for a geometric generic point
$\QQ(S)\subset \CC$. Grothendieck-Katz's $p$-curvature conjecture
predicts that if for all closed points $s$ of some non-trivial
open $U\subset S$, the $p$-curvature of $(E_S,\nabla_S)\times_S s$
is zero, then $(E,\nabla)$ is trivialized by a finite \'etale
cover of $X$ (see e.g. \cite[Conj.3.3.3]{An}). Little is known
about it. N. Katz proved it for Gau{\ss}-Manin connections
\cite{Ka},  for $S$ finite over $\Spec \ZZ[\frac{1}{N}]$ (i.e., if
$X$ can be defined over a number field), D. V. Chudnovsky and G.
V. Chudnovsky in \cite{CC} proved it in the rank $1$ case and Y.
Andr\'e in \cite{An} proved it in case the Galois differential Lie
algebra of $(E, \nabla)$ at the generic point of $S$ is solvable
(and for extensions of connections satisfying the conjecture).
More recently, B. Farb and M. Kisin \cite{FK} proved it for
certain locally symmetric varieties $X$. In general, one is
lacking methods to think of the problem.

\medskip

Y. Andr\'e in \cite[II]{An} formulated the following  equal
characteristic $0$ analog of the conjecture:  if $X\to S$ is a
smooth morphism of smooth connected varieties defined over a
characteristic $0$ field $k$, then if $(E_S,\nabla_S)$ is a
relative  integrable connection such that for all closed points
$s$ of some non-trivial open $U\subset S$, $(E_S,\nabla_S)\times_S
s$ is trivialized by a finite \'etale cover of $X\times_S s$, then
$(E,\nabla)|_{X_{\bar \eta}}$ should be  trivialized by a finite
\'etale cover, where $\bar \eta$ is a geometric generic point and
$X_{\bar \eta}=X\times_S \bar \eta$. So the characteristic $0$
analogy to {integrable} connections is simply { integrable}
connections, and to the $p$-curvature condition is the
trivialization of the connection by a finite \'etale cover.
 He proved it \cite[Prop.~7.1.1]{An}, using Jordan's theorem  and
Simpson's moduli of flat connections.

\medskip

It is tempting to formulate an equal characteristic $p>0$  analog of Y.
Andr\'e's theorem. A main feature of integrable connections over a field $k$ of
characteristic
$0$ is
that they form an abelian, rigid, $k$-linear tensor category. In
characteristic $p>0$, the category of bundles with an integrable  connection is
only $\mathcal{O}_{X^{(1)}}$-linear, where
$X^{(1)}$ is the relative Frobenius twist of $X$, and the notion is too weak.
On the other hand, in characteristic $0$, the category of
bundles with a flat connection is the same as the category of ${\mathcal
O}_X$-coherent $\mathcal{D}_{X}$-modules. In characteristic $p>0$,
${\mathcal
O}_X$-coherent $\mathcal{D}_{X}$-modules over a smooth variety $X$ defined over
 a field $k$ form an abelian, rigid, $k$-linear tensor category (see
\cite{Gi}).  It is equivalent to the category of stratified
bundles. It bares strong analogies with the category of bundles
with an integrable  connection in characteristic $0$. For example,
if $X$ is projective smooth over an algebraically closed field,
the triviality of the \'etale fundamental group forces all such
${\mathcal O}_X$-coherent $\mathcal{D}_{X}$-modules to be trivial
(\cite{EM}).

\medskip

So we raise the {\bf question 1}: let $f: X\to S$ be a smooth
projective morphism of smooth connected varieties,  defined over
an algebraically closed characteristic $p>0$ field, let
$(E,\nabla)$ be a stratified bundle relative to $S$, such
that for all closed point $s$ of some non-trivial open $U\subset
S$, the stratified bundle $(E,\nabla)|_{X_s}$ is trivialized by a
finite \'etale cover of $X_s:=X\times_S s$. Is it the case that
the stratified bundle $(E,\nabla)|_{X_{\bar \eta}}$ is trivialized
by a finite \'etale cover of $X_{\bar \eta}$?.

\medskip

In this form, this is not true. Y. Laszlo \cite{Ls} constructed a one
dimensional non-trivial family of bundles over a curve over $\FF_2$ which is
fixed by the square of Frobenius, as  a (negative) answer to a question of J. de
Jong concerning the behavior of representations of the \'etale fundamental
group over a finite field $\FF_q, \  q=p^a$, with values in $GL(r, \FF((t)))$,
where $\FF\supset \FF_2$ is a finite extension. In fact, Laszlo's example yields
also a counter-example to the question as stated above. We explain this in
Sections \ref{s:strat} and  \ref{s:laszlo} (see
Corollary~\ref{counterexample}).  We remark that
if $E$ is a bundle on $X$, such that the  bundle $E|_{X_s}$ is
stable, numerically flat (see Definition~\ref{numflat})  and moves in the
moduli, then $E_{\bar \eta}$ cannot be trivialized by a finite \'etale cover
(see Proposition~\ref{Esnault}).
In contrast, we show that  if the family $X\to S$ is trivial (as it is in
Laszlo's example), thus $X=Y\times_k S$, if $k$ is algebraically closed, and if
$(F^n_{Y}\times {\rm identity}_s)^*(E)|_{Y\times_k s} \cong E|_{Y\times_k s}$ for all closed points
$s$ of
some non-trivial open in $S$ and some fixed natural number $n$, then the moduli
points of $E|_{Y\times_k s}$ are constant (see Proposition~\ref{reason}). Here
$F_Y: Y\to Y$ is the absolute Frobenius of $Y$. In Laszlo's example, one does
have  $(F^2_{Y}\times {\rm
identity}_s)^*(E)|_{Y\times_k s} \cong E|_{Y\times_k s}$ but only over $k=\FF_2$
(i.e., $S$ is also defined over $\FF _2$).
 When one extends the family to the algebraic closure of $\FF_2$, to go from
the
absolute Frobenius over $\FF_2$, that is the relative Frobenius over $k$, to
the absolute one, one needs to replace the power  $2$ with a higher power
$n(s)$, which depends on the field of definition
of $s$, and is not bounded.

\medskip

So we modify question 1 in {\bf question 2}:  let $f: X\to S$ be a
smooth projective morphism of smooth connected varieties,  defined
over an algebraically closed characteristic field $k$ of
characteristic $p>0$, let $E$ be a bundle such that for all closed
points $s$ of  some non-trivial open $U\subset S$, the  bundle
$E|_{X_s}$ is trivialized by a finite  Galois \'etale cover of
$X_s:=X\times_S s$ of order prime to $p$.
 Is it the case that the
 bundle $E|_{X_{\bar \eta}}$ is
trivialized
by a finite \'etale cover of $X_{\bar \eta}$?.

\medskip

The answer is nearly yes: it is the case if $k$ is  not
algebraic over its prime field (Theorem~\ref{GK-etale} 2)).  If
 $k=\bar{\FF}_p$, it might be wrong (Remarks~\ref{Remarks}
2), but what remains true is that there exists a finite \'etale
cover of $X_{\bar \eta}$ over which the pull-back of $E$ is a
direct sum of line bundles (Theorem~\ref{GK-etale} 1)). The idea
of the proof is borrowed from  the proof of Y. Andr\'e's theorem
\cite[Thm~7.2.2]{An}. The assumption on the degrees of the Galois
covers of $X_s$ trivializing $E|_{X_s}$ is necessary (as follows
from  Laszlo's example) and it allows us to apply Brauer-Feit's
theorem \cite[Theorem]{BF} in place of Jordan's theorem  used by
Andr\'e. However, there is no direct substitute for Simpson's
moduli spaces of flat bundles. Instead, we use the moduli spaces
constructed in \cite{La1} and we carefully analyze subloci
containing the points of interest, that is the numerically flat
bundles. The necessary material needed on moduli is gathered in
Section \ref{moduli}.

\medskip

Finally we raise the general {\bf question 3}: let $f: X\to S$ be a smooth
projective
morphism
of smooth connected varieties,  defined over an
algebraically closed characteristic $p>0$ field, let $(E, \nabla)$ be a
stratified bundle relative to $S$, such that
for all closed points $s$ of  some non-trivial
open $U\subset S$, the   bundle $E|_{X_s}$ is trivialized
by a finite  Galois \'etale cover of $X_s:=X\times_S s$ of order prime to $p$.
 Is it the case that the
 bundle $E|_{X_{\bar \eta}}$ is
trivialized
by a finite \'etale cover of $X_{\bar \eta}$?.

\medskip

 We give the following not quite complete answer. If the rank of
$E$ is $1$, (in which case the assumption on the degrees of the
Galois covers is {\it automatically fulfilled}), then the answer
is yes provided {$S$ is projective}, and for any $s\in U$,
$\Pic^\tau(X_s)$ is reduced (see
Theorem~\ref{stratified-line-bundles}). The proof relies on (a
variant of) an idea of M. Raynaud \cite{Ra}, using the height
function associated to a symmetric line bundle (that is the reason
for  our assumption on $S$) on the abelian scheme and its dual, to
show that an infinite Verschiebung-divisible point has height
equal to $0$ (Theorem~\ref{thm:height}) . If
$E$ has any rank, then the answer is yes if $k$ is not
{$\bar{\FF}_p$} (Theorem~\ref{stratified-vb} 2)).  In general,
there is a prime to $p$-order Galois cover of $X_{\bar \eta}$ such
that the pull-back of $E$ becomes a sum of stratified line bundles
(Theorem~\ref{stratified-vb} 1)).

\medskip

{\it Acknowledgements:} The first author thanks  Michel Raynaud
for the fruitful discussions in November  2009, which are
reflected in \cite{Ra} and in Section~\ref{Raynaud}. The first
author thanks Johan de Jong for a beautiful discussion in November
2010 on the content of \cite{EM}, where she suggested question $1$
to him, and where he replied that Laszlo's example should
contradict this, and that this should be better understood. The
second author would like to thank Stefan Schr\"oer for destroying
his naive hopes concerning N\'eron models of Frobenius twists of
an abelian variety. We thank Damian R\"ossler for discussions on
$p$-torsion on abelian schemes over functions fields. {We thank
the referee of a first version of the article. He/she explained to
us that the dichotomy in Theorem~\ref{GK-etale} 2) and in
Theorem~\ref{stratified-vb} 2) should be $\bar{\FF}_p$ or not
rather that countable or not, thereby improving our result.}

\section{Preliminaries on relative stratified sheaves} \label{s:strat}

Let $S$ be a scheme of characteristic $p$ (i.e., $\sO_S$ is an
$\FF _p$-algebra). By $F_S^r :S\to S$ we denote the $r$-th
\emph{absolute Frobenius morphism} of $S$ which corresponds to the
$p^r$-th power mapping on $\sO _S$.

If $X$ is an $S$-scheme, we denote  by $X^{(r)}_S$ the fiber
product of $X$ and $S$ over the $r$-th
Frobenius morphism of $S$. If it is clear with respect to which structure $X$ is
considered, we simplify the notation to
 $X^{(r)}$. Then
the $r$-th absolute Frobenius
morphism of $X$ induces the \emph{relative Frobenius morphism}
$F^r_{X/S}: X\to X^{(r)}$. In particular, we have the following
commutative diagram:
$$ \xymatrix{ &X\ar[rd]\ar@(ur,ul)[rr]^{\mbox{\scriptsize
$F_{X}^r$}}\ar[r]_{\mbox{\tiny $F_{X/S}^r$}}&
{X^{(r)}}\ar[d] \ar[r]_{W_X^r} & X\ar[d]\\
&&S\ar[r]_{\mbox{\scriptsize $F_S^r$}}&S\\}$$ which defines
$W_{X/S}^r: X^{(r)}\to X$.

Making $r=1$ and replacing   $X$ by $X^{(i)}$, this induces the
similar diagram
$$ \xymatrix{ &X^{(i)} \ar[rd]\ar@(ur,ul)[rr]^{\mbox{\scriptsize
$F_{X^{(i)}}$}}\ar[r]_{\mbox{\tiny $F_{X^{(i)}/S}$}}&
{X^{(i+1)}}\ar[d] \ar[r]_{W_{X^{(i)}}} & X^{(i)}\ar[d]\\
&&S\ar[r]_{\mbox{\scriptsize $F_S$}}&S\\}$$

We assume that $X/S$ is smooth. A \emph{relative stratified sheaf}
on $X /S$ is a sequence $\{E_i, \sigma_i\} _{i\in \NN}$ of locally
free coherent $\sO_{X^{(i)}}$-modules $E_i$ on $X ^{(i)}$ and
isomorphisms $\sigma _i: F_{X ^{(i)}/S}^*E_{i+1}\to E_i$ of
$\sO_{X^{(i)}}$-modules. A \emph{morphism of relative stratified
sheaves} $\{\alpha _i\}:\{E_i, \sigma_i\}\to \{E_i', \sigma_i'\}$
is a sequence of $\sO_{X^{(i)}}$-linear maps $\alpha_i: E_i\to
E_i'$ compatible with the $\sigma_i$, that  is  such that $\sigma
_i'\circ F^*_{X^{(i)}/S}\alpha _{i+1}=\alpha _i \circ \sigma_i$.

This forms a category ${\sf Strat}(X/S)$, which is contravariant
for morphisms $\varphi: T\to S$: to $\{E_i,\sigma_i\} \in {\sf
Start}(X/S)$ one assigns $\varphi^*\{E_i,\sigma_i\}\in {\sf
Strat}(X\times_S T/T)$ in the obvious way: $\varphi$ induces
$1_{X^{(i)}}\times \varphi: X^{(i)}\times_S T \to X^{(i)}$ and  $
(\varphi^*\{E_i,\sigma_i\})_i= \{(1_{X^{(i)}}\times \varphi)^*E_i,
      (1_{X^{(i)}}\times \varphi)^*(\sigma_i)      \}.     $

If $S=\Spec k$ where $k$ is a field, ${\sf Strat}(X/k)$ is an
abelian, rigid, tensor category. Giving a rational point $x\in
X(k)$ defines a fiber functor via $\omega_x: {\sf Strat}(X/k)\to
{\sf Vec}_k, \ \omega_x( \{E_i, \sigma_i\}    )= (E_0)|_{x}$ in
the category of finite dimensional  vector spaces over $k$, thus a
$k$-group scheme $\pi({\sf Strat}(X/k), \omega_x)={\rm
Aut}^{\otimes}(\omega_x)$. Tannaka duality implies that ${\sf
Strat}(X/k)$ is equivalent via $\omega_x$ to the representation
category of $\pi({\sf Strat}(X/k), \omega_x)$ with values in ${\sf
Vec}_k$. For any object $\EE:=\{E_i, \sigma_i\} \in {\sf
Strat}(X/k)$, we define its {\it monodromy group}
 to be the $k$-affine group scheme $\pi(\langle \EE \rangle, \omega_x)$, where
$\langle \EE \rangle \subset {\sf Strat}(X/k)$ is the full
subcategory spanned by $\EE$. This is the image of  $\pi({\sf
Strat}(X/k), \omega_x)$ in $GL(\omega_x(\EE))$
(\cite[Proposition~2.21 a)]{DM}). We denote by
$\mathbb{I}_{X/k}\in {\sf Strat}(X/k)$ the trivial object, with
$E^i=\sO_{X^{(i)}}$ and $\sigma_i={\rm Identity}$.
\begin{Lemma} \label{lem:finiteness}
With the notation above
\begin{itemize}
\item[1)] If  $h:  Y\to X$ is a finite \'etale cover such that
$h^*\EE$ is trivial, then $h_*\mathbb{I}_{Y/k}$ has finite
monodromy group and one has a faithfully flat homomorphism
$\pi(\langle h_* \mathbb{I}_{Y/k}\rangle, \omega_x) \to
\pi(\langle \EE \rangle, \omega_x) $. Thus in particular, $\EE$
has finite monodromy group as well.
\item[2)] If  $\EE\in {\sf Strat}(X/k)$ has finite monodromy group,
then there  exists a  $\pi(\langle \EE \rangle, \omega_x)$-torsor
$h: Y\to X$  such that $h^*\EE$ is trivial in ${\sf Strat}(Y/k)$.
Moreover, one has an isomorphism $\pi(\langle h_*
\mathbb{I}_{Y/k}\rangle, \omega_x) { \xrightarrow{\cong}}
\pi(\langle \EE \rangle, \omega_x) $.
\end{itemize}
\end{Lemma}
\begin{proof}
We first prove 2).  Assume $\pi(\langle \EE \rangle, \omega_x)=:G$
is a finite group scheme over $k$. One applies Nori's method
\cite[Chapter~I,~II]{No}: the regular representation of $G$ on the
affine $k$-algebra $k[G]$ of regular function defines the Artin
$k$-algebra $k[G]$ as a $k$-algebra object of the representation
category of $G$ on finite dimensional $k$-vector spaces, (such
that $k\subset k[G]$ is the maximal trivial subobject). Thus by
Tannaka duality, there is an object $\mathbb{A}=(A^i, \tau_i)\in
{\sf Strat}(X/k)$, which is  an $\mathbb{I}_{X/k}$-algebra
object, (such that $\mathbb{I}_{X/k} \subset \mathbb{A}$ is the
maximal trivial subobject). We define $h_i: Y_i=\Spec_{X^{(i)}}
A^i \to X^{(i)}$. Then the isomorphism $\tau_i$ yields an
$\sO_{X^{(i)}}$-isomorphism between $Y^{(i)}\xrightarrow{h^{(i)}}
X^{(i)}$ and  $Y_i\xrightarrow{h_{i}} X^{(i)}$, (see, e.g.,
\cite[Expos\'e~XV,~\S~1,~Proposition 2]{SGA5}), and via this
isomorphism, $\mathbb{A}$ is isomorphic to $h_*\mathbb{I}_{Y/k}$.
On the other hand, $\omega_x(\EE)$ is a sub $G$-representation  of
$k[G]^{\oplus n}$ for some $n\in \mathbb{N}$, thus $\EE\subset
\mathbb{A}^{\oplus n}$ in ${\sf Strat}(X/k)$, thus there is an
inclusion $\EE\subset (h_* \mathbb{I}_{Y/k})^{\oplus n}$ in ${\sf
Strat}(X/k)$, thus $h^*\EE\subset (h^*h_*
\mathbb{I}_{Y/k})^{\oplus n}$ in ${\sf Strat}(Y/k)$. Since
$(h^*h_* \mathbb{I}_{Y/k})$ is isomorphic to $\oplus_{{\rm
length}_k k[G]} \mathbb{I}_{Y/k}$ in ${\rm Strat}(Y/k)$  (recall
that by \cite[Proposition~13]{dS}, $G$ is an \'etale group
scheme), then $h^*\EE$ is isomorphic to $\oplus_r
\mathbb{I}_{Y/k}$, where $r$ is the rank of $\EE$. This shows the
first part of the statement, and shows the second part as well:
indeed, $\EE$ is then a subobject of $\oplus_r
h_*\mathbb{I}_{Y/k}$, thus $\langle \EE\rangle\subset \langle
h_*\mathbb{I}_{Y/K}\rangle$ is a full subcategory. { One applies
\cite[Proposition~2.21~a)]{DM} to show that the induced
homomorphism $   \pi(\langle h_* \mathbb{I}_{Y/k}\rangle,
\omega_x) \to  \pi(\langle \EE \rangle, \omega_x)=G $ is
faithfully flat. So $ \pi(\langle h_* \mathbb{I}_{Y/k}\rangle,
\omega_x)$ acts on $\omega_x(h_*\mathbb{I}_Y)=k[G]$ via its
quotient $G$ and the regular representation $G\subset GL(k[G])$.
Thus the homomorphism is an isomorphism.}

We show 1). Assume  that there is a finite \'etale cover $h: Y\to
X$ such that $h^*\EE$ is isomorphic  in ${\sf Strat}(Y/k)$ to
$\oplus_r \mathbb{I}_{Y/k}$ where $r$ is the rank of $\EE$. Then
$\EE\subset \oplus_r h_*\mathbb{I}_{Y/k}$, thus $\pi(\langle
h_*\mathbb{I}_{Y/k} \rangle , \omega_x) \to \pi(\langle \EE
\rangle, \omega_x) $ is faithfully flat \cite[loc. cit.]{DM},
 so we are reduced to showing that { $\langle
h_*\mathbb{I}_{Y/k}\rangle $ has finite
monodromy.  But, by the same argument as on $\EE$,  any of its objects of rank
$r'$   lies in $\oplus_{r'} h_*\mathbb{I}_{Y/k}$.}
So we apply \cite[Proposition~2.20~a)]{DM}
to conclude that the monodromy of $h_*\mathbb{I}_{Y/k}$ is finite.
\end{proof}

\begin{Corollary} \label{cor:fin_bc}
With the notations as in \ref{lem:finiteness}, if $\EE\in {\sf
Strat}(X/k)$ has finite monodromy group, then for any field
extension $K\supset k$,  $\EE\otimes K \in {\sf Strat}(X\otimes
K/K) $ has finite monodromy group.

\end{Corollary}

Let $E$ be an $\sO _{X}$-module. We say that \emph{$E$ has a
stratification relative to $S$} if there exists a relative
stratified sheaf $\{E_i, \sigma_i\}$ such that $E_0=E$.

 Let us consider the special case $S =\Spec k$, where $k$
is a perfect field, and $X/k$ is smooth.
 An \emph{(absolute) stratified sheaf} on
$X$ is a sequence $\{E_i, \sigma_i\} _{i\in \NN}$ of coherent
$\sO_{X}$-modules $E_i$ on $X$ and isomorphisms $\sigma _i:
F_{X}^*E_{i+1}\to E_i$ of $\sO_{X}$-modules.

As $k$ is perfect, the $W_{X^{(i)}}$ are isomorphisms, thus giving
an absolute stratified sheaf is equivalent to giving a stratified
sheaf relative to $\Spec k$.

 We now go back to the general case and we assume that $S$
is an integral $k$-scheme, where $k$ is a field. Let us set
$K=k(S)$ and let $\eta : \Spec K\to S$ be the generic point of
$S$. Let us fix an algebraic closure $\bar K$ of $K$ and let $\bar
\eta$ be the corresponding generic geometric point of $S$.

By contravariance, a relative stratified sheaf $\{E_i, \sigma_i\}$
on $X /S$ restricts
 to a relative stratified sheaf
$\{E_i, \sigma_i\}|_{X_s} $ in fibers $X_s$ for  $s$ a point of
$S$. We are interested in the relation between $\{E_i,
\sigma_i\}|_{X_{\bar \eta}} $ and $\{E_i, \sigma_i\}|_{X_s} $ for
closed points $s\in |S|$. More precisely, we want to understand
under which assumptions the finiteness of $\langle \{E_i,
\sigma_i\}|_{X_s}\rangle$  for all closed points $s\in |S|$
implies the finiteness of $\langle \{E_i, \sigma_i\}|_{X_{\bar
\eta}}\rangle$. { Recall that finiteness} of $\EE \subset {\sf
Strat}(X_s)$ means that all objects of $\langle \EE \rangle$ are
subquotients in ${\sf Strat}(X_s)$ of { direct sums of a single
object,  } which is equivalent to saying that after the choice of
a rational point, the monodromy group of $\EE$ is finite
(\cite[Proposition~2.20 (a)]{DM}).

\medskip

Let $X$ be a smooth variety defined over $\FF_q$ with $q=p^r$. For
all $n\in \NN\setminus \{0\}$, one has the commutative diagram
\ga{1}{\xymatrix{\ar[dr] X \ar@(ur,ul)[rr]^{(F_X^r)^n=F_X^{rn}}
\ar[r]_{F^{rn}_{X/\FF_q}}
& \ar[d] X^{(rn)} \ar[r]_{W^{rn}_{X/\FF_{q}}} & \ar[d] X\\
& \Spec \FF_q \ar[r]_{F_{\FF_{q}}^{rn}=\id} & \Spec \FF_{q}} }
which allows us to identify $X^{(rn)}$ with $X$ (as an $\FF
_q$-scheme).

Let $S$ be an $\FF_q$ connected scheme, with field of constants
$k$, i.e. $k$ is the normal closure of $\FF_q$ in $H^0(S, \sO_X)$.
We define $X_S:=X\times_{\FF_q} S$.
\begin{Proposition} \label{prop1}
 Let $E$ be a vector bundle on $X_S$. Assume that there exists a positive
integer $n$ such that we have an isomorphism \ga{2}{ \tau:
((F^r\times_{\FF_q}\id_S)^n)^*E\simeq E .} Then $E$ has a natural
stratification $\EE _{\tau}=\{E_i, \sigma _i \}, \ E_0=E$ relative
to $S$.

\end{Proposition}
\begin{proof}
  We define
\ga{3}{  E_{rn}=  (W^{rn}_{X/\FF_{q}}\times _{\FF _q} \id_S)^*E .
} Then we use the factorization \ga{4}{\xymatrix{\ar[drrrr]
X\ar[r]^{F_{X/\FF_q}} & \ar[drrr]
X^{(1)}\ar[r]^{F_{X^{(1)}/\FF_q}} & \cdots \ar[r]&
X^{(rn-1)}\ar[dr]
\ar[r]^{F_{X^{(rn-1)}/\FF_q}} & \ar[d] X^{(rn)}\\
 & & & & \Spec \FF_q
} } of $F_{X/\FF _q}^{rn}$ and we define \ga{5}{
E_{nr-1}=(F_{X^{(rn-1)}/\FF_q}\times_{\FF_q} \id_S)^*E_{rn},
\ldots, E_1=(F_{X^{(1)}/\FF_q}\times_{\FF_q} \id _S)^*E_2} with
identity isomorphisms $\sigma_{nr-1},\dots, \sigma _1$. Then we
use the isomorphism $\tau$ to define \ga{6}{\sigma _0: E\simeq
(F_{X/\FF_q}\times_{\FF _q} \id _S)^*E_1.} Assume we constructed
the bundles $E_i$ on $X^{(i)}$ for all $i\le arn $ for some
integer $a\ge 1$.

We now replace the diagram \eqref{1} by the diagram
\ga{7}{\xymatrix{\ar[dr] X^{(arn)}
\ar@(ur,ul)[rr]^{(F_{X^{(arn)}}^r)^n}
\ar[r]_{F^{rn}_{X^{(arn)}/\FF_q}}
& \ar[d] X^{((a+1)rn)} \ar[r]_{W^{rn}_{X^{(arn)}/\FF_{q}}} & \ar[d] X^{(arn)}\\
&\Spec \FF_q \ar[r]_{F_{\FF_{q}}^{rn}=1} & \Spec \FF_{q}} } We
then define \ga{8}{E_{(a+1)rn}=
(W^{rn}_{X^{(arn)}/\FF_{q}}\times_{\FF_q}\id_S)^*E_{arn} } (which
is equal to $E$ under identification of $X^{(arn)}$ with $X$).
Then we use the factorization \ga{9}{\xymatrix{\ar[drrrr]
X^{(arn)}\ar[r]^{F_{X^{(arn)}/\FF_q}} & \ar[drrr]
X^{(arn+1)}\ar[r]^{F_{X^{(arn+1)}/\FF_q}} & \cdots \ar[r]&
X^{((a+1)rn-1)}\ar[dr]
\ar[r]^{F_{X^{((a+1)rn-1)}/\FF_q}} & \ar[d] X^{((a+1)rn)}\\
 & & & & \Spec \FF_q
} } of $F_{X^{(arn)}/\FF _q}^{rn}$ to define \ml{10}{
E_{(a+1)rn-1}=(F_{X^{((a+1)rn-1)}/\FF_q}\times_{\FF_q}
\id_S)^*E_{(a+1)rn}, \ldots,\\
E_{arn+1}=(F_{X^{(arn+1)}/\FF_q}\times_{\FF_q} \id_S)^*E_{arn+2} }
with identity isomorphisms $\sigma_{(a+1)nr-1},\dots, \sigma
_{arn+1}$. Then we again use $\tau$ to define
\ga{11}{\sigma_{arn}: E_{arn}\simeq
(F^{rn}_{X^{(arn)}/\FF_q})^*E_{arn+1}.}
\end{proof}

\medskip

The above construction and \cite[Proposition 1.7]{Gi} imply

\begin{Proposition} \label{prop3}
Assume in addition to \eqref{2} that  $X$ is proper and
$\FF_q\subset k\subset \bar \FF_q$. Fix a rational point $x\in
X_S(k)$. Then for any closed point $s\in |S|$, the Tannaka group
scheme $\pi(\EE _{\tau _s}, \omega_{x\otimes_k k(s)})$ of $\EE
_{\tau _s}:=\EE_{\tau}|_{X_s}$ over the residue field $k(s)$ of
$s$ is finite.
\end{Proposition}
\begin{proof}
The bundle $E$ is base changed of a bundle $E^0$ defined over
$X\times_{\FF_q} S_0$ for some form $S_0$ of $S$ defined over a
finite extension $\FF_{q^a}$ of $\FF_q$ such that $x$ is base
change of an $\FF_{q^a}$-rational point $x_0$ of $X\times_{\FF_q}
S_0$. We can also assume that $\tau$ comes by base change from
$\tau _0: ((F^r\times_{\FF_q}\id_{S_0})^n)^*E^0\simeq E^0$.
Proposition~\ref{prop1} yields then a relative stratification
$\EE^0_{\tau_0}=(E_i^0, \sigma_i^0)$ of $E^0$ defined over
$\FF_{q^a}$, with $E_i=E_i^0\otimes_{\FF_q^a} k$. A closed point
$s$ of $S=S_0\otimes_{\FF_q^a} k$ is a base change of some closed
point $s_0$ of $S_0$ of degree $b$ say over $\FF_{q^a}$.  By
Corollary~\ref{cor:fin_bc} we just have to show  that
$\pi(\EE_{(\tau_0) _{s_0}}, \omega_{x_0 \otimes_{\FF_{q^a}}
k(s_0)})$ is finite. So we assume that $k=\FF_{q^a}, \ S=S_0, \
s=s_0$. The underling bundles of  $\EE_{\tau}$ and $\EE_{\tau
^{m}}$ are by construction all isomorphic for $m=ab$. Thus by
\cite[Proposition~1.7]{Gi}, $\EE_{\tau}\simeq \EE_{\tau ^{m}}$ in
${\sf Strat}(X/k)$. But this implies that
$F^{mn}_{X\times_{\FF_{q^a}} s}(\EE_{\tau _s})\cong \EE _{\tau _s}
$. Thus $E$ is algebraically trivializable on the Lang torsor $h:
Y\to X\times_{\FF_{q^a}} \FF_{q^m}$ and the bundles $E_i$ are
trivializable on $Y\times_{X\times_{\FF_{q^a}} \FF_{q^m}}
X^{(i)}=Y^{(i)}/\FF_{q^m}$. Thus the stratified bundle $h^*\EE
_{\tau}$ on $Y$ relative to $\FF_{q^m}$ is trivial. We apply
Lemma~\ref{lem:finiteness} to finish the proof.
\end{proof}

\section{\'Etale trivializable bundles}

Let $X$ be a smooth projective variety over an algebraically
closed field $k$. Let $F_X: X\to X$ be the absolute Frobenius
morphism.

A locally free sheaf on $X$ is called \emph{\'etale trivializable}
if there exists a finite \'etale covering of $X$ on which $E$
becomes trivial.

Note that if $E$ is \'etale trivializable then it is numerically
flat {(see Definition~\ref{numflat} and the subsequent discussion)}. In
particular, stability and semistability for such bundles
are independent of a polarization (and Gieseker and slope
stability and semistability are equivalent). More precisely, such
$E$ is stable if and only if it does not contain any locally free
subsheaves of smaller rank and degree $0$ (with respect to some or equivalently
to any polarization).

\begin{Proposition} \emph{(see \cite{LSt})} \label{La-St}
If there exists a positive integer $n$ such that  $(F_X^n)^*E\simeq E$ then
$E$ is \'etale
trivializable. Moreover, if $k=\bar \FF _p$ then $E$ is \'etale trivializable if
and only
if there exists a positive integer $n$ and an isomorphism
$(F_X^n)^*E\simeq E$.
\end{Proposition}

\begin{Proposition} \emph{(see \cite{BD})} \label{prop:BD}
If there exists a finite degree $d$ \'etale Galois covering $f:
Y\to X$ such that $f^*E$ is trivial and $E$ is stable, then one
has an isomorphism $\alpha: (F_X^d)^*E\simeq E$.
\end{Proposition}

As a corollary we see that a line bundle on $X/k$ is \'etale
trivializable if and only if it is torsion of order prime to $p$.
One implication follows from the above proposition. The other one
follows from the fact that $(F_X^d)^*L\simeq L$ is equivalent to
$L^{\otimes (p^d-1)}\simeq \O_X$ and for any integer $n$ prime to
$p$ we can find $d$ such that $ p^d-1$ is divisible by $n$.

We recall that if $E$ is any vector bundle on $X$ such that there
is a $d\in \mathbb{N}\setminus \{0\}$ and an isomorphism $\alpha:
(F_X^d)^*(E)\cong E$, then $E$ carries an {\it absolute}
stratified structure $\EE_{\alpha}$, i.e. a stratified structure
relative to $\FF_p$ by the procedure of Proposition~\ref{prop1}.
On the other hand, any stratified stratified structure $\{E_i, \sigma_i\}$
relative to $\FF_p$ induces in an obvious way a stratified structure relative
to $k$:  the absolute Frobenius $F^n_X: X\to X$ factors through
$W^n_{X/k}: X^{(n)} \to X$, so $\{ (W^n_{X/k})^* E_n, (W^n_{X/k})^*\sigma_n\}$
is the relative stratified structure, denoted by  $\EE_{\alpha/k}$.
 Proposition~\ref{prop:BD}
together with Lemma~\ref{lem:finiteness} 2) show
\begin{Corollary} \label{control}
 Under the assumptions of Proposition~\ref{prop:BD}, we can take $d=
{\rm length}_k k[\pi(\langle \EE_{\alpha/k}  \rangle, \omega_x)]$.
\end{Corollary}

Let us also recall that there exist examples of \'etale
trivializable bundles such that $(F_X^n)^*E\not\simeq E$ for every
positive integer $n$ (see Laszlo's example in \cite{BD}).

\begin{Proposition} \label{Deligne} \emph{(Deligne;  see \cite[3.2]{Ls})}
Let $X$ be an { $\FF _{p^n}$}-scheme. If $G$ is a connected linear
algebraic group defined over a finite field $\FF _{p^n}$ then the
embedding $G(\FF _{p^n})\hookrightarrow G$ induces an equivalence
of categories between the category of $G(\FF _{p^n})$-torsors on
$X$ and $G$-torsors $P$ over $X$ with an isomorphism
$(F_X^n)^*P\simeq P$.
\end{Proposition}

In particular, if $G$ is a connected reductive algebraic group
defined over an algebraically closed field $k$ and $P$ is a
principal $G$-bundle on $X/k$ such that there exists an
isomorphism $(F_X^n)^*P\simeq P$  for some natural number $n>0$,
then there exists a Galois \'etale cover $f:Y\to X$ with Galois
group $G(\FF_{p^n})$ such that $f^*P$ is trivial.  Indeed, every
reductive group has a $\ZZ$-form so we can use the above
proposition.

\section{Preliminaries on relative moduli spaces of
sheaves\label{moduli}}

Let $S$ be a scheme of finite type over a universally Japanese
ring $R$. Let $f: X\to S$ be a projective morphism of $R$-schemes
of finite type with geometrically connected fibers and let
$\O_X(1)$ be an $f$-very ample line bundle.

A \emph{family of pure Gieseker semistable sheaves on the fibres
of} $X_T=X\times_S T\to T$ is a $T$-flat coherent
$\O_{X_T}$-module $E$ such that for every geometric point $t$ of
$T$ the restriction of $E$ to the fibre $X_t$ is pure (i.e., all
its associated points have the same dimension)  and Gieseker
semistable  (which is semistability with respect to the growth of
the Hilbert polynomial  of subsheaves defined by $\O_X(1)$ (see
\cite[1.2]{HL}). We introduce an equivalence relation $\sim $ on
such families in the following way. $E\sim E'$ if and only if
there exist filtrations $0=E_0\subset E_1\subset ... \subset
E_m=E$ and $0=E_0'\subset E_1'\subset ... \subset E_m'=E'$ by
coherent $\O_{X_T}$-modules such that $\oplus
_{i=0}^m{E_i/E_{i-1}}$ is a family of pure Gieseker semistable
sheaves on the fibres of $X_T$ and there exists an invertible
sheaf $L$ on $T$ such that {$\oplus _{i=1}^m{E_i'/E_{i-1}'}\simeq
\left(\oplus _{i=1}^m{E_i/E_{i-1}}\right)\otimes _{\O_T} L$.}

Let us define the moduli functor
$$\M _{P}(X/S) : (\hbox{\rm Sch/}S) ^{o}\to \hbox{Sets} $$
from the category of locally noetherian schemes over $S$ to the
category of sets by
$$\M _{P}(X/S) (T)=\left\{
\aligned
&\sim\hbox{equivalence classes of families of pure Gieseker}\\
&\hbox{semistable sheaves on the fibres of }T\times_S X\to T,\\
&\hbox{which have Hilbert polynomial }P.\\
\endaligned
\right\} .$$

Then we have the following theorem (see {
\cite[Theorem~0.2]{La1}}).

\begin{Theorem}  Let us fix a polynomial $P$. Then there
exists a projective $S$-scheme $M_P(X/S)$ of finite type over $S$
and a natural transformation of functors
$$\theta :\M _P(X/S)\to \Hom _S (\cdot, M_P(X/S)),$$
which uniformly corepresents the functor $\M _{P}(X/S)$. For every
geometric point $s\in S$ the induced map $\theta (s)$ is a
bijection. Moreover, there is an open scheme
$M^{s}_{X/S}(P)\subset M_{P}(X/S)$ that universally corepresents
the subfunctor of families of geometrically Gieseker stable
sheaves.
\end{Theorem}

Let us recall that $M_P(X/S)$ {\it uniformly corepresents}
$\M_P(X/S)$ means that for every flat base change $T\to S$ the
fiber product $ M_P(X/S) \times _S T$ corepresents the fiber
product functor $\Hom_S (\cdot , T)\times _{\Hom _S(\cdot
,S)}\M_P(X/S)$.   For the notion of corepresentability, we refer
to \cite[Definition 2.2.1]{HL}. In general, for every $S$-scheme
$T$ we have a well defined morphism $M_P(X/S)\times _S T\to M_P(X_T/T)$
which for a geometric point $T=\Spec k(s)\to S$ is bijection
on points.

 The moduli space $M_P(X/S)$ in general depends on the choice of
polarization $\O_X(1)$.

 \begin{Definition} \label{numflat}
Let $k$ be a field and let $Y$ be a projective $k$-variety. A
coherent $\O_Y$-module $E$ is called \emph{numerically flat}, if
it is locally free and both $E$ and {its dual $E^*=\mathcal{H}om
(E, \O_Y)$} are numerically effective on $Y\otimes \bar k$, where $\bar k$ is
an algebraic closure of $k$.
\end{Definition}

Assume that $Y$ is smooth. Then a numerically flat sheaf is
strongly slope semistable  of degree $0$ with respect to
any polarization (see \cite[Proposition 5.1]{La2}). But such a
sheaf has  a filtration with quotients which are numerically flat
and slope stable (see \cite[Theorem 4.1]{La2}). Let us recall that
a slope stable sheaf is Gieseker stable and any extension of
Gieseker semistable sheaves with the same Hilbert polynomial is
Gieseker semistable. Thus a numerically flat sheaf is Gieseker
semistable with respect to any polarization.

Let $P$ be the Hilbert polynomial of  the trivial sheaf of rank
$r$. In case $S$ is a spectrum of a field we  write $M_X(r)$ to
denote the subscheme of the moduli space $M_P(X/k)$ corresponding
to locally free sheaves. For a smooth projective morphism $X\to S$
we also define the moduli subscheme $M(X/S, r)\to S$ of the
relative moduli space $M_P(X/S)$ as a union of connected
components which contains points corresponding to numerically flat
sheaves of rank $r$. Note that in positive characteristic
numerical flatness is not an open condition.      More precisely,
on  a smooth projective variety $Y$  with an ample divisor $H$, a
locally free sheaf with {\it numerically trivial Chern classes,
that is with Chern classes  $c_i$ in the Chow group of codimension
$i$ cycles intersecting trivially $H^{{\rm dim}(Y)-i}$ for all
$i\ge 1$,} is numerically flat if and only if it is strongly slope
semistable (see \cite[Proposition~5.1]{La2}).

By definition for every family $E$ of pure Gieseker semistable
sheaves on the fibres of $X_T$ we have a well defined morphism
$\varphi _E=\theta ([E]): T\to M_P(X/S)$, which we call a
\emph{classifying morphism.}

\begin{Proposition}\label{semi-trivializability}
Let $X$ be a smooth projective variety defined over an
algebraically closed field $k$ of positive characteristic. Let $S$
be a $k$-variety and let $E$ be a rank $r$ locally free sheaf on
$X\times _k S$ such that for every $s\in S(k)$ the restriction
$E_s$ is Gieseker semistable with numerically trivial Chern
classes. Assume that the classifying morphism $\varphi _E:S\to
M_X(r)$ is constant and for a dense subset $S'\subset S(k)$ the
bundle $E_s$ is \'etale trivializable for $s\in S'$. Then
$E_{\bar \eta}$  is \'etale trivializable.
\end{Proposition}

\begin{proof}
If $E_s$ is stable for some $k$-point $s\in S$ then there exists
an open neighbourhood $U$ of $\varphi _E(s)$, a finite \'etale
morphism $U'\to U$ and a locally free sheaf $\U$ on $X\times _k
U'$ such that the pull backs of $E$ and $\U$ to $X\times _k
(\varphi _E ^{-1}(U)\times _U U')$ are isomorphic (this is called
existence of a universal bundle on the moduli space in the \'etale
topology). But $\varphi _E(S)$ is a point, so this proves that
there exists a vector bundle on  $X$ such that $E$ is its pull
back by the projection $X\times _k S\to X$. In this case the
assertion is obvious.

Now let us assume that $E_s$ is not stable for all $s\in S(k)$. If
$0=E_0^s\subset E_1^s\subset ... \subset E_m^s=E_s$ is a
Jordan--H\"older filtration  (in the category of slope
semistable torsion free sheaves), then by assumption the
isomorphism classes of semi-simplifications $\oplus
_{i=1}^m{E_i^s/E_{i-1}^s}$ do not depend on $s\in S(k)$. Let
$(r_1, ..., r_m)$ denote the sequence of ranks of the components
$E_i^s/E_{i-1}^s$ for some $s\in S(k)$. Since there is only
finitely many such sequences (they differ only by permutation), we
choose some permutation that appears for a dense subset
$S''\subset S'$.

Now let us consider the scheme of relative flags $f:\Flag (E/S;
P_1,..., P_m)\to S$, where $P_i$ is the Hilbert polynomial of
$\O_X^{r_i}$. By our assumption the image of $f$ contains $S''$.
Therefore by Chevalley's theorem it contains an open subscheme $U$
of $S$. Let us recall that the scheme of relative flags $\Flag
(E|_{X\times _k U}/U; P_1,..., P_m)\to U$ is projective. In
particular,  using Bertini's theorem ($k$ is algebraically closed)
we can find a generically finite morphism $W\to U$ factoring
through this flag scheme. Let us consider pull back of the
universal filtration $0=F_0\subset F_1\subset ...\subset F_m=E_W$
to $X\times _k W$. Note that the quotients $F^i=F_i/F_{i-1}$ are
$W$-flat and by shrinking $W$ we can assume that they are families
of Gieseker stable locally free sheaves (since by
assumption $F^i_s$ is Gieseker stable  and locally free for
some points $s\in W(k)\cap S'$). This and the first part of the
proof implies that $E_{\bar \eta}$ has a filtration by subbundles
such that the associated graded sheaf is \'etale trivializable. By
Lemma \ref{splitting} this implies that $E_{\bar \eta}$ is \'etale
trivializable.
\end{proof}

\section{Laszlo's example} \label{s:laszlo}

\medskip

Let us describe  Laszlo's example of a line in the moduli space of
bundles on a curve fixed by the second Verschiebung morphism (see
\cite[Section 3]{Ls}).

Let us consider a smooth projective genus $2$ curve $X$ over
$\FF_2$ with affine equation
$$y^2+x(x+1)y=x^5+x^2+x.$$
In this case the moduli space $ M_X(2, \O_X)$ of rank $2$ vector
bundles on $X$ with trivial determinant is an $\FF _2$-scheme
isomorphic to $\PP ^3$. The pull back of bundles by the relative
Frobenius morphism defines the Verschiebung map
$$V: M_{X^{(1)}}(2, \O_{X^{(1)}})\simeq \PP ^3\dashrightarrow M_X(2,{\O_X})\simeq \PP ^3$$
which in appropriate coordinates can be described as
$$[a:b:c:d]\to [a^2+b^2+c^2+d^2:ab+cd:ac+bd:ad+bc].$$
The restriction of $V$ to the line $\Delta\simeq \PP ^1$ given by
$b=c=d$ is an involution and it can be described as $[a:b]\to
[a+b:b]$.

Using a universal bundle on the moduli space (which exists locally
in the \'etale topology around points corresponding to stable
bundles) and taking a finite covering $S\to \Delta$  we obtain the
following theorem:

\begin{Theorem} \label{Laszlo} \emph{(\cite[Corollary 3.2]{Ls})}
There exist a smooth quasi-projective curve $S$ defined over some
finite extension of $\FF _2$ and  a locally free sheaf $E$ of rank
$2$ on $X\times  S$ such that $(F^2\times \id _S)^*E\simeq E$,
 $\det E\simeq \O_{X\times S}$ and the classifying morphism
$\varphi _E: S\to M_X(2, \O_X)$ is not constant. Moreover, one can
choose $S$ so that $E_s$ is stable for every closed point $s$ in
$S$.
\end{Theorem}

Now note that the map $(F_X)^*: M_X(2,{\O_X})\dashrightarrow
M_X(2,{\O_X})$ defined by pulling back bundles by the absolute
Frobenius morphism can be described on $\Delta$ as  $[a:b]\to
[a^2+b^2:b^2]$. In particular, the map $(F_X^{2n})^*|_{\Delta}$ is
described as $[a:b]\to [a^{2n}, b^{2n}]$. It follows that if a
stable bundle $E$ corresponds to a modular point of $\Delta
(\FF_2^{n})\backslash \Delta (\FF_2^{n-1})$ (or, equivalently, $E$
is defined over $\FF_{2^{n}}$) then $(F_X^{2n})^*E\simeq E$ and
$(F_X^{m})^*E\not\simeq E$ for $0<m<2n$.

This implies that for $k=\bar \FF_2$ and for every $s\in S(k)$,
the bundle $E_s$ which is the restriction to $X\times_{\FF_2} s$
of the  bundle $E$ from Theorem \ref{Laszlo}, is \'etale
trivializable.

\medskip

Let $X,S$ be varieties defined over an algebraically closed field
$k$ of positive characteristic. Assume that $X$ is projective. Let
us set $K=k(S)$. Let $\bar \eta$ be a generic geometric point of
$S$.

\begin{Proposition} \label{Esnault}
Let $E$ be a   bundle on   $X_S=X\times _k S\to S $  which is
numerically flat on the closed fibres of $X_S=X\times _k S\to S $.
Assume that for some $s\in S$ the bundle $E _s$ is stable and the
classifying morphism $\varphi _E: S\to M_X(r)$ is not constant.
Then $E_{\bar \eta}=E|_{X_{\bar \eta}}$ is  not \'etale
trivializable.
\end{Proposition}

\begin{proof}
Assume that there exists a finite \'etale cover $\pi':Y' \to
X_{\bar \eta}$ such that $(\pi')^*E _{\bar \eta}\simeq \O_
{Y'}^r$.  As $k$ is algebraically closed, one has the base change
$\pi_1(X)\xrightarrow{\cong} \pi_1(X_{\bar K})$ for the \'etale
fundamental group {(\cite[Exp.~X,~Cor.1.8]{SGA1}),} so there  exists a
finite \'etale cover $\pi:
Y\to X$ such that $\pi'=\pi\otimes \bar K$. Hence there exists  a
finite morphism  $T\to U$ over some open subset $U$ of $S$, such
that $\pi_T^*(E_T)$ is trivial where $\pi_T=\pi \times _k \id _T:
Y\times _kT\to X\times _k T$ and $E_T=$pull back by $X\times_k T \to
X\times_k U$ of $E|_{X\times_kU}$.

So for any $k$-rational point $t \in T$, one has $\pi^*
E_t\subset \O_Y ^r$, where $r$ is the rank of $E$. Hence
$E_t\subset \pi_*\pi^* E_t \subset \pi_*\O_{Y}^r$, i.e., all the
bundles $E_t$ lie in one fixed bundle $\pi_*\O_Y^r$.

Since $\pi$ is \'etale, the diagram
$$ \xymatrix{ &Y\ar[d]^{\pi}\ar[r]^{F_Y}& Y\ar[d]^{\pi}\\
&X\ar[r]^{F_X}&X }$$ is cartesian (see, e.g., { \cite[Exp. XIV,
{\S}1, Prop.~2]{SGA5}}). Since $X$ is smooth, $F_X$ is flat. By
flat base change we have isomorphisms $F_X^*(\pi_*\O_Y)\simeq
\pi_*(F_Y^*\O_Y)\simeq \pi_*\O_Y$. In particular, this implies
that $\pi_*\O_Y$ is strongly semistable of degree $0$. Therefore
if $E_t$ is stable then it appears as one of the factors in a
Jordan--H\"older filtration of $\pi_*\O_Y$. Since the direct sum
of factors in a Jordan--H\"older filtration of a semistable sheaf
does not depend on the choice of the filtration, there are only
finitely many possibilities for the isomorphism classes of stable
sheaves $E_t$ for  $t\in T(k)$.

It follows that in $U\subset S$ there is an infinite sequence of
$k$-rational points $s_i$ with the property that $E_{s_i}$ is
stable (since stability is an open property) and $E_{s_i}\cong
E_{s_{i+1}}$. This contradicts our assumption that the classifying
morphism $\varphi _E$ is not constant.
\end{proof}

\medskip

\begin{Corollary} \label{counterexample}
There exist smooth curves $X$ and $S$ defined over an algebraic
closure $k$ of $\FF _2$ such that $X$ is projective and there
exists a locally free sheaf $E$ on $X\times _k S\to S$ such
that for every $s\in S(k)$, the bundle $E_s$ is \'etale
trivializable but $E_{\bar \eta}$ is not \'etale trivializable.
Moreover, on $E$ there exists a structure of a relatively
stratified sheaf $\EE$ such that for every $s\in S(k)$, the bundle
$\EE_s$ has finite monodromy but the monodromy group of $\EE_{\bar
\eta}$ is infinite.
\end{Corollary}

The second part of the corollary follows from Proposition
\ref{prop1}. The above corollary should be compared to the
following fact:

\begin{Proposition} \label{reason}
Let $X$ be a projective variety defined over an algebraically
closed field $k$ of positive characteristic. Let $S$ be a
$k$-variety and let $E$ be a rank $r$ locally free sheaf on
$X\times _k S$. Assume that there exists a positive integer $n$
such that for every $s\in S(k)$ we have $(F_X^n)^*E_s\simeq E_s$,
where $F_X$ denotes the absolute Frobenius morphism. Then the
classifying morphism $\varphi _E:S\to M_X(r)$ is constant and
 $E_{\bar \eta}$ is \'etale trivializable.
\end{Proposition}

\begin{proof}
By  Proposition \ref{La-St}, if $(F_X^n)^*E_s\simeq E_s$ then
there exists a finite \'etale Galois cover $\pi _s: Y_s\to X$ with
Galois group $G=\GL _r (\FF _{p^n})$ such that $\pi _s^*E_s$ is
trivial (in this case it is essentially due to Lange and Stuhler;
see \cite{LSt}). This implies that $E_s\subset (\pi_s)_*\pi
_s^*E_s\simeq ((\pi_s)_* \O_Y)^{\oplus r}$ and hence $\gr _{JH}
E_s \subset (\gr _{JH}(\pi_s)_*\O_Y)^{\oplus r}$.

Since $X$ is proper, the \'etale fundamental group of $X$ is
topologically finitely generated and hence there exists only
finitely many finite \'etale coverings of $X$ of fixed degree (up
to an isomorphism). This theorem is known as the Lang--Serre
theorem (see \cite[Th\'eor\`eme 4]{LS}). Let $\cS$ be the set of
all Galois coverings of $X$ with Galois group $G$. Then for every
closed $k$-point $s$ of $S$ the semi-simplification of $E_s$ is
contained in $(\gr _{JH}\alpha _*\O_Y)^{\oplus r}$ for some
$\alpha \in \cS$. Therefore there are only finitely many
possibilities for images of $k$-points $s$ in $M_X(r)$. Since $S$
is connected, it follows that $\varphi _E:S\to M_X(r)$ is
constant.

 The remaining part of the proposition follows from
Proposition \ref{semi-trivializability}.

\end{proof}

\medskip

Note that by Proposition~\ref{Esnault} together with
Corollary~\ref{control},  the monodromy groups of $E_s$ in
Theorem~\ref{Laszlo}
 for
$s\in S(k)$ are not uniformly bounded. In fact, only if $k$ is an
algebraic closure of a finite field do we know that the
monodromy groups of $E_s$ are finite because then $E_s$ can be
defined over some finite subfield of $k$ and the isomorphism
$(F^2)^* E_s\simeq E_s$ implies that for some $n$ we have
$(F_X^n)^*E_s\simeq E_s$ (see the paragraph following Theorem
\ref{Laszlo}).

Moreover, the above proposition shows that in
Theorem~\ref{Laszlo}, we cannot hope to replace $F$ with the
absolute Frobenius morphism $F_X$.

\section{Analogue of the Grothendieck-Katz conjecture
in positive equicharacteristic}

As Corollary \ref{counterexample} shows,   the positive
equicharacteristic version of the Grothendieck--Katz conjecture
which requests
 a relatively stratified bundle to have finite monodromy group on the
geometric generic fiber once it does on all closed fibers, does
not hold in general. But one can still hope that it holds for a
family of bundles coming from representations of the prime-to-$p$
quotient of the \'etale fundamental group. In this section we
follow Andr\'e's approach \cite[Th\'eor\`eme~7.2.2]{An} in the
equicharacteristic zero case to show that this is indeed the case.

\medskip

Let $k$ be an algebraically closed field of positive
characteristic $p$. Let $f: X\to S$ be a smooth projective
morphism of $k$-varieties (in particular, integral $k$-schemes).
Let $\eta$ be the generic point of $S$.  { In particular, {
$X_{\bar \eta}$} is smooth (see \cite[Defn~1.1]{SGA1}). }

\begin{Theorem} \label{GK-etale}
Let $E$ be a locally free sheaf of rank $r$ on $X$. Let us assume
that there exists  a dense subset $U\subset S(k)$ such that for
every $s$ in $U$, there is a finite Galois \'etale covering
$\pi_s: Y_s\to X_s$ of Galois group of order prime-to-$p$ such
that $\pi_s^*(E_s)$ is trivial.
\begin{itemize}
 \item[1)]  Then   there exists a finite Galois \'etale covering $\pi
_{\bar
\eta}: Y_{\bar \eta}\to X_{\bar \eta}$  of order
prime-to-$p$ such that $\pi _{\bar \eta}^*E_{\bar \eta}$ is a
direct sum of line bundles.
\item[2)] If  $k$ is not algebraic over its prime field and $U$ is open in $S$, then
$E_{\bar \eta}$ is \'etale trivializable on a finite \'etale
cover $Z_{\bar \eta}\to X_{\bar \eta}$ which factors as a Kummer
(thus finite abelian of order prime to $p$) cover $Z_{\bar
\eta}\to Y_{\bar \eta}$ and a Galois cover $Y_{\bar \eta}\to
X_{\bar \eta}$ of order prime to $p$.
\end{itemize}

\end{Theorem}

\begin{proof}
Without loss of generality, shrinking $S$ if necessary, we may
assume that  $ S$ is smooth. Moreover, by
passing to a finite cover of $S$ and replacing $U$ by its inverse image,
we can assume that $f$ has a
section $\sigma: S\to X$.

By assumption for every  $s\in U$ there exists a finite \'etale
Galois covering $\pi _s: Y_s\to X_s$ with Galois group $\Gamma _s$
of order prime-to-$p$ and such that $\pi_s^*E_s$ is trivial. To
these data one can associate a representation $\rho _s : \pi^{p'}
_1(X_s, \sigma (s))\to \Gamma _s\subset \GL _r(k)$ of the
prime-to-$p$ quotient of the \'etale fundamental group.

By the Brauer--Feit version of Jordan's theorem (see
\cite[Theorem]{BF}) there exist a constant $j(r)$ such that
$\Gamma _s $ contains an abelian normal subgroup $A_s$ of index
$\le j(r)$ (here we use assumption that the $p$-Sylow subgroup of
$\Gamma _s$ is trivial).

For a $k$-point $s$ of $S$ we have a homomorphism of
specialization
$$\alpha _s: \pi _1 (X_{\bar \eta}, \sigma (\bar\eta ))\twoheadrightarrow \pi _1 (X_s, \sigma (s)),$$
which induces an isomorphism of the prime-to-$p$ { quotients of
the \'etale fundamental groups. }

 So for every $s\in U$ we can define the composite
morphism
\ga{}{ \tilde\rho _s: \pi _1 ^{p'}(X_{\bar \eta},
\sigma (\bar\eta )) \xrightarrow{\alpha_s} \pi _1^{p'} (X_s,
\sigma (s)) \xrightarrow{\rho_s} \Gamma _s\twoheadrightarrow
\Gamma_s/A_s. \notag }

 Let $K$ be the kernel of the canonical homomorphism $\pi_*:
\pi _1 (X, \sigma (\bar \eta)) {\longrightarrow} \pi _1 (S, \bar
\eta)$, {  let $K^{p'}$  be its maximal pro-$p'$-quotient. Then by
}
 \cite[Exp. XIII, Proposition 4.3 and Exemples 4.4]{SGA1},
{  one has $ K^{p'}=\pi _1 ^{p'}(X_{\bar \eta}, \sigma ({\bar
\eta}))$, the maximal pro-$p'$-quotient of   $\pi _1(X_{\bar
\eta}, \sigma ({\bar \eta}))$, } and one has a short exact
sequence
$$ \{1\}\to \pi _1 ^{p'}(X_{\bar \eta}, \sigma ({\bar \eta})) {\longrightarrow}
\pi _1' (X, \sigma (\bar\eta )) \mathop{\longrightarrow}^{\pi_*}
\pi _1 (S, \bar\eta)\to \{1\},$$ { where $\pi _1' (X, \sigma
(\bar\eta )) $ is defined as the push-out of $\pi _1 (X, \sigma
(\bar\eta )) $ by $K\to K^{p'}$. }

Since $X_{\bar \eta}$ is proper, $\pi _1 (X_{\bar \eta}, \sigma
(\bar\eta ))$ is topologically finitely generated. Therefore $\pi
_1^{p'} (X_{\bar \eta}, \sigma (\bar\eta ))$ is also topologically
finitely generated and hence it contains only finitely many
subgroups of indices $\le j(r)$. Let $G$ be the intersection of
all such subgroups  in $\pi
_1^{p'} (X_{\bar \eta}, \sigma (\bar\eta ))$. It is a normal subgroup of
finite index.  Since $\ker
(\tilde\rho _s)$ is a normal subgroup of index $\le j(r)$ in $\pi
_1^{p'} (X_{\bar \eta}, \sigma (\bar\eta ))$ we have
$$G\subset \bigcap _{ s\in U}\ker (\tilde\rho _s).$$

Now let us consider the commutative diagram
$$\xymatrix{
& &  \pi _1 (X_{\bar \eta}, \sigma (\bar \eta)) \ar[r] \ar[d] &
\pi
_1 (X, \sigma (\bar\eta)) \ar[r] \ar[d] & \pi _1 (S, \bar\eta ) \ar[r]\ar[d] &\{1\} \\
&\{1\}\ar[r]&  \pi _1^{p'} (X_{\bar\eta}, \sigma (\bar\eta))
\ar[r] & \pi _1^{'} (X, \sigma (\bar\eta))\ar[r] & \pi _1 (S,
\bar\eta ) \ar[r]& \{ 1 \} }$$ Then  $G \cdot \sigma _* (\pi _1
(S, \bar \eta))\subset \pi _1' (X, \sigma (\bar \eta))$ is a
subgroup of finite index. It is open by the Nikolov--Segal theorem
\cite[Theorem 1.1]{NS}. So  the pre-image $H$ of  this subgroup
under the quotient homomorphism $\pi _1 (X, \sigma (\bar\eta))\to
\pi _1^{'} (X, \sigma (\bar\eta))$ defines a finite \'etale
covering $h: X' \to X$.

Let us take $s\in S(k)$. Since the composition
$$H\subset \pi _1
(X, \sigma (\bar \eta)) \to \pi _1 (X, \sigma (s)) \to \pi_1(S,
s)$$ is surjective, the geometric fibres of $X'\to S$ are
connected. Let us choose a $k$-point in $X'$ lying over $\sigma
(s)$. By abuse of notation we call it $\sigma '(s)$. Similarly,
let us choose a geometric point $\sigma '(\bar\eta)$ of $X'_{\bar
\eta}$ lying over $\sigma (\bar\eta)$. Then for any  $s\in U$ we
have the following commutative diagram:
$$\xymatrix{
&  \pi _1^{p'} (X'_{\bar \eta}, \sigma '(\bar
\eta))\ar@(ur,ul)[rr]^{0} \ar[r]^{h _*}\ar[d]^{ \simeq} &  \pi
_1 ^{p'}(X_{\bar \eta}, \sigma (\bar\eta)) \ar[r] \ar[d]^{\alpha_s
\ \simeq} & \pi _1^{p'} (X_{\bar\eta}, \sigma (\bar\eta))/ G
\ar[d] \\
&  \pi _1^{p'} (X'_{s}, \sigma '(s)) \ar[r]^{h_*} &  \pi _1^{p'}
(X_{s}, \sigma (s))\ar[r] & \Gamma _s/A_s}$$ This diagram shows
that $ \pi _1^{p'} (X'_{s}, \sigma '(s))\to \Gamma _s$ factors
through $A_s$ and hence $E'_s=(h ^*E)_s$ is trivialized by a
finite \'etale Galois covering { $\pi _s': Y_s'\to X_s'$} with an
abelian Galois group  of order prime to $p$, which is a subgroup
of $A_s$. Since
$$E'_s\subset (\pi _s')_* (\pi _s')^*E'_s \simeq ((\pi _s')_*\O_{Y'_s})^{\oplus
r},$$ and $(\pi _s')_*\O_{Y'_s}$ is a direct sum of torsion line
bundles of orders prime to $p$, it follows that  for every $s\in
U$ the bundle $E'_s$ is also a direct sum of torsion line bundles
of order prime to $p$.

 We consider the union $M(X'/S, r)$ of the components of
$M_P(X'/S)$ containing moduli points of numerically flat bundles,
as defined in Section \ref{moduli}.  Let us consider the
$S$-morphism $\psi: M(X'/S,1)^{\times _S r}\to M(X'/S)$  given by
$([L_1],..., [L_r])\to [\oplus L_i]$ (in fact we give it by this
formula on the level of functors; existence of the morphism
follows from the fact that moduli schemes corepresent these
functors). The bundle $E'$ gives us a section $\tau : S\to
M(X'/S,r)$, and by the above  for every $k$-rational point $s$ of
$U$, the point $\tau (s)$ is contained in the image of $\psi$.
Therefore $\tau (S)$ is contained in the image of $\psi$  as
$\psi$ is projective (thus proper).

Let us consider the fibre product
$$
\xymatrix{
&M(X'/S,1)^{\times _{S} r} \times _{M(X'/S,r)} S\ar[r]\ar[d]& S\ar[d]^{\tau}\\
&M(X'/S,1)^{\times _{S} r} \ar[r]&  M(X'/S,r)\\
}
$$
Let us recall that in positive characteristic the canonical map
$M(X'\times _SS '/S',r) \to  M(X'/S, r)\times _{S}S'$ need not be
an isomorphism (although it is an isomorphism for $r=1$). Anyway
we can find an \'etale morphism $S'\to S$ over some non-empty open
subset of $S$, such that there exists a map $\upsilon: S'\to
M(X'\times _SS '/S',1)^{\times _{S'} r} $ which composed with $
M(X'\times _SS '/S',1)^{\times _{S'} r}\to M(X'\times _SS '/S',r)
\to  M(X'/S, r)$ gives the composition of $S'\to S$ with $\tau$.
This shows that the pull back $E''$ of $E'$ to $X'\times _S S'$
has a filtration whose quotients { are line bundles which are } of
degree $0$ on the fibres of $X'\times _S S'\to S'$. Now let us
note the following lemma:

\begin{Lemma}{\label{splitting}}
Let $f: X\to S$ be a projective morphism of $k$-varieties. Let
$0\to G_1\to G\to G_2\to 0$ be a sequence of locally free sheaves
on $X$. Assume that  there exists a dense subset $U\subset S(k)$
such that for each $s\in U$ this sequence splits after restricting
to $X_s$. Then it splits on the fibre $X_{\eta}$ over the generic point $\eta$
of $S$.
\end{Lemma}

\begin{proof}
By shrinking $S$ if necessary, we may assume that $S$ is affine
and  the relative cohomology sheaf $R^1p_* \cHom (G_2, G_1)$ is
locally free. The above short exact sequence defines a class
$\lambda \in {\rm Ext} ^1 (G_2,G_1)\simeq  H^0(S, R^1f_* \cHom
(G_2, G_1))$, such that $\lambda (s)=0$ for every $k$-rational
point $s$ of $U$. It follows that $\lambda =0$ and hence the
sequence is split over the generic point of $S$.
\end{proof}

\medskip

Now let us note that on a smooth projective variety every short
exact sequence of the form $0\to G_1\to G\to G_2\to 0$ in which
$G$ is a direct sum of line bundles of degree $0$ and $G_2$ is a
line bundle of degree $0$ splits. { So the  filtration of $E''$ restricted
to the closed fibers splits}. Therefore the above lemma and
easy induction show that $E''_{\eta'}$  is a direct sum of line
bundles, { where $\eta'$ is the generic point of $S'$}. This shows the
first part of the theorem.

\medskip

To prove the second part of the theorem, we may assume that $U=S$.
 Let us take a line bundle $L$ on $X$ such that for every
$k$-rational point $s$ the line bundle $L_s$ is \'etale
trivializable. We need to prove that there exists a positive
integer $n$ prime to $p$ and such that $L^{\otimes n}_{\eta}\simeq
\O_{X_{\eta}}$.

\medskip

{We thank the referee for showing us the following lemma.}

\begin{Lemma} \label{dJ}
Let $g: A\to S$ be an abelian scheme and let $\sigma$ be a section
of $g$ such that for all $s\in S(k)$,  $ \sigma(s)$ is torsion of
order prime to $p$. Then $\sigma$ is torsion of order prime to
$p$.
\end{Lemma}

\begin{proof}
We may assume that $S$ is normal and affine. Let us choose a
subfield $k'\subset k $ that is finitely generated and
transcendental over $\FF _p$ and such that $A\to S$ and $\sigma$
come {by} base change $\Spec k\to \Spec k'$ from an abelian scheme
$g':A'\to S'$ and a section $\sigma'$ defined over $k'$.
 Let $m>1$
be prime to $p$ and let $\Gamma$ be the subgroup $A'(S')\cap
[m]^{-1}(\ZZ .\sigma ')$ of $A'(S')$.  Then
$\Gamma$ is a finitely generated group. Note that assumptions of
N\'eron's specialization theorem \cite[Chapter 9, Theorem 6.2]{L}
are satisfied and therefore there exists a Hilbert set $\Sigma$ of
points $s'\in S'$ for which the specialization map $A'(S')\to
A'_{s'}(k(s'))$ is injective on $\Gamma$. Since the Hilbert subset
$\Sigma \subset S'$ contains infinitely many closed points (see
\cite[Chapter 9, Theorems 5.1, 5.2 and 4.2]{L}), there is a closed
point $s\in S$ { the  image of which} in $S'$ lies in $\Sigma$.
The specialization of $\ZZ .\sigma$ at $s$ is injective and hence
$\sigma$ is torsion of order dividing the order of $\sigma (s)$, {
which is prime to $p$.}
\end{proof}

\medskip

Let us first assume that $X\to S$ is of relative dimension $1$. By
passing to a finite cover of $S$ we can assume that $f$ has a
section. The relative Picard scheme $A=\Pic ^0 (X/S)\to S$ is
smooth. Using the above lemma to the section corresponding to the
line bundle $L$ we see that there exists some positive integer $n$
prime to $p$ and a line bundle $M$ on $S$ such that $L^{\otimes
n}\simeq f^*M$. In particular, $L^{\otimes n}_{\eta}\simeq
\O_{X_{\eta}}$.

Now we use induction on the relative dimension of $f:X\to S$ to
prove  the theorem in the general case. Note that our assumptions
imply that $L_{\bar \eta}$ is numerically flat and therefore the
family $\{L^{\otimes n}_{\bar \eta}\} _{n\in \ZZ}$ is bounded.
Thus  for any sufficiently ample divisor $H$ on $X_{\bar \eta}$ we
have $H^1(X_{\bar \eta}, L_{\bar \eta}^{\otimes n}(-H))=0$ for all
integers $n$. We consider such an $H$ which is defined over
$\eta$.

Using Bertini's theorem we can find a very ample divisor $Y\subset
X$ in the linear system $|H|$ such that $f|_Y: Y\to S$ is smooth
(possibly after shrinking $S$) and such that for every positive
integer $n$ we have $H^1(X_{\eta}, L^{\otimes
n}(-Y)|_{Y_{\eta}})=0$. Indeed, shrinking $S$ and
using semicontinuity of cohomology, we may assume that $H$ is
defined over $S$, that  the function $\dim H^0(X_s,
\O_{X_s}(H))$ is constant and $S$ is affine. Let us choose a
$k$-rational point $s$ in $S$. Then by Grauert's
theorem (see \cite[Chapter III, Corollary 12.9]{Ha}) the
restriction map
$$H^0(X, \O_X(H))\to H^0(X_{s}, \O_{X_s}( H))$$
is surjective. By Bertini's theorem in the
linear system $| \O_{X_s}( H)|$ there exists a smooth divisor. By
the above we can lift it to a divisor $Y\subset X$, which after
shrinking $S$ is the required divisor.

Applying our induction assumption to $L|_Y$ on $Y\to X$ we see
that there exists a positive integer $n$  prime to $p$ such that
$(L|_Y)_{\eta} ^{\otimes n}\simeq \O_{Y_\eta}$. Using the short
exact sequence
$$0\to L_{\eta}^{\otimes n}(-Y_{\eta})\to L_{\eta}^{\otimes n}\to (L_Y^{\otimes n})_{\eta}\to 0$$
we see that the map
$$H^0(X_{\eta}, L_{\eta}^{\otimes n})\to H^0(Y_{\eta}, (L_Y^{\otimes n})_{\eta})$$
is surjective. In particular, $L_{\eta}^{\otimes n}$ has a section
and hence it is trivial.
\end{proof}

\begin{Remarks} \label{Remarks}
\begin{enumerate}
\item Laszlo's example shows that the first part of the theorem is false if one does not
assume that orders of the monodromy groups of $E_s$ are prime to
$p$ (in this example $E_{\bar \eta}$ is a stable rank $2$ vector
bundle). Note that in this example, $E$ has { even the richer structure of
 a}
relatively stratified bundle (see Proposition~\ref{prop1}).
\item
Let $E$ be  a supersingular elliptic curve defined over
$k=\bar\FF_p$. Let $M$ be a line bundle of degree $0$ and of infinite
order on $E_{\overline{\FF _p (t)}}$. Then one can find a smooth
curve $S$ defined over $k$ such that there exists a line bundle
$L$ on $X=S\times _kE\to S$  such that $L_{\bar \eta}\simeq M$. In
this example  the line bundle $L_s$ is torsion for every
$k$-rational point $s$ of $S$ { as it is defined over a finite
field}. Since $E$ is a supersingular elliptic curve, there are no
torsion line bundles of order  divisible by $p$. So in this case
all line bundles $L_s$ for $s\in S(k)$ are \'etale trivializable
(and the monodromy group has order prime to $p$).

This shows that the second part of Theorem \ref{GK-etale} is no
longer true if $k$ is an algebraic closure of a finite
field.
\end{enumerate}
\end{Remarks}

\medskip

Let us keep the notation from the beginning of the section, i.e.,
$k$ is an algebraically closed field of positive characteristic
$p$ and $f: X\to S$ is a smooth projective morphism of
$k$-varieties (in particular connected) with geometrically
connected fibers. For simplicity, we also assume that $f$ has a
section $\sigma: S\to X$.

\begin{Lemma}
Let $E$ be a locally free sheaf on $X$. If there exists  a  point $s_0\in S(k)$ such that
$E_{s_0}$ is numerically flat then $E_{\bar \eta}$ is also numerically flat.
In particular, if there exists  a point  $s_0\in S(k)$ such that there is a finite covering
$\pi_{s_0}: Y_{s_0}\to X_{s_0}$ such that $\pi_{s_0}^*(E_{s_0})$ is trivial,
then $E_{\bar \eta}$ is also numerically flat.
\end{Lemma}

\begin{proof}
Let us fix a relatively ample line bundle.
If $E_{s_0}$ is numerically flat then it is strongly semistable with
numerically trivial Chern classes (see \cite[Proposition 5.1]{La2}).  Since  $E$
is $S$-flat, the restriction of $E$ to any fiber  has  numerically trivial Chern
classes (as intersection numbers remain constant on fibres). Now note that
for any $n$ the sheaf $(F^n_{X_{s_0}/k})^*E_{s_0}$ is slope semistable. Since
slope semistability is an open property, it follows that
$(F^n_{X_{\eta}/K})^*E_{\eta}$ is also slope semistable. By \cite[Corollary 1.3.8]{HL} it follows that
$(F^n_{X_{\bar \eta}/\bar K})^*E_{\bar \eta}$ is also slope semistable.
Thus $E_{\bar \eta}$ is strongly semistable with vanishing Chern classes
and hence it is numerically flat  by \cite[Proposition 5.1]{La2}.
\end{proof}

\medskip

Let us recall that numerically flat sheaves on a proper
$k$-variety $Y$ form a Tannakian category.  A rational point  $y\in
Y(k)$ neutralizes it. Thus we can define
\emph{S-fundamental group scheme of $Y$  at the point $y$} (see
\cite[Definition 6.1]{La2}).
For a  numerically flat sheaf $E$ on $Y$, we
consider the Tannaka $k$-group $\pi _S(\langle E\rangle , y):={\rm
Aut}^{\otimes}(\langle E\rangle , y)\subset \GL (E_y)$, where now
$\langle E\rangle$ is the full tensor subcategory of numerically flat bundles
spanned by $E$.
We call it the \emph{S-monodromy group scheme}.
Using this language we can reformulate Theorem \ref{GK-etale} in the following way
(for simplicity we reformulate only the second part of the theorem).

\begin{Theorem}
Let $E$ be  an $S$-flat family of numerically flat sheaves on the
fibres of $X\to S$. Let us assume that $k$ is not algebraic over
its prime field and there exists a non-empty open subset
 $U\subset S(k)$ such that for
every $s$ in $U$, the S-monodromy group scheme $\pi _S(\langle
E_s\rangle , \sigma (s))$ is finite \'etale of order prime-to-$p$.
Then $\pi _S(\langle E_{\bar \eta}\rangle , \sigma (\bar \eta))$
is also  finite \'etale.

\end{Theorem}

\section{Verschiebung divisible points on abelian varieties: on the
theorem by M. Raynaud} \label{Raynaud}

Let $K$ be an arbitrary field of positive characteristic $p$ and
let $A$ be an abelian variety defined over $K$. The multiplication
by $p^n$ map $[p^n]: A\to A$ factors through the relative
Frobenius morphism $F_{A/K}^n: A\to A^{(n)}$ and hence  defines
the \emph{ Verschiebung morphism} $V^n: A^{(n)} \to A$ such that
$V^nF_{A/K}^n=[p^n]$.

\begin{Definition}
A $K$-point $P$ of  $A$ is said to be \emph{V-divisible} if for
every positive integer $n$ there exists a $K$-point $P_n$ in
$A^{(n)}$ such that $V^n(P_n)=P$.
\end{Definition}

 Let $T$ be an integral noetherian separated scheme of
dimension $1$ with field of rational functions $K$. Let us recall
that a smooth, separated group scheme of finite type $\A\to T$ is
called a \emph{N\'eron model} of $A$ if the general fiber of
$\A\to T$ is isomorphic to $A$ and for every smooth morphism $X\to
T$, a morphism $X_K\to \A_K$ extends  {(then uniquely)} to a
$T$-morphism $X\to \A$.

\medskip
Assume that the base field $K$ is the function field of a normal
projective variety $S$ defined over a field $k$ of positive
characteristic $p$.

We say that $A$ has \emph{a good reduction} at a codimension $1$
point $s\in S$ if the N\'eron model of $A$ over $\Spec \O _{S,s}$
is an abelian scheme { (the usual definition is slightly different
as it assumes that the identity component of the special fibre of
the N\'eron model is an abelian variety; it is equivalent to the
above one by \cite[7.4, Theorem 5]{BLR}).}
 We say that $A$
has \emph{ potential good reduction} at a codimension $1$ point
$s\in S$ if there exists a finite Galois extension $K'$ of $K$
such that if $S'$ is the normalization of $S$ in $K'$ then
$A_{K'}$ has good reduction at every codimension $1$ point $s'\in
S'$ lying over $s$.

We say that $A$ has \emph{(potential) good reduction} if it has
(potential) good reduction at every codimension $1$ point of $S$.
Assume
that $A$ has good reduction at every codimension $1$ point of $S$.
Then there exists a {\it big open} subset $U\subset S$ (i.e., the
codimension of the complement of $U$ in $S$  is $\ge 2$) and an
abelian $U$-scheme $\A \to U$. Note that the group $A(K)$ of
$K$-points of $A$ is isomorphic via the restriction map to the
group of rational sections $ U\dashrightarrow \A $ of $\A\to U$
defined over some big open subset of $U$.  The section
corresponding to $P\in A(K)$ will be denoted by $\ti P: U
\dashrightarrow \A.$

Let $c\in \Pic A$ be a class of a line bundle $L$. By the
theorem of the cube $c$, satisfies the following equality:
$$m_{123}^*c-m_{12}^*c-m_{13}^*c-m_{23}^*c+m_{1}^*c+m_{2}^*c+m_{3}^*c=0,$$
where $m_I$ for $I\subset \{1,2,3\}$ is the map $A\times _KA\times
_KA\to A$ defined by addition over the factors in $I$.  (In
particular, $m_i$ is the $i$-th projection).
Combining \cite[Chapter III,~3.1]{MB} (relying on
\cite[Chapter II,~Proposition~1.2.1]{MB}), the line bundle $L \in \Pic(A)$ extends
uniquely (at least if we fix a rigidification) to a line bundle $\ti L$ over
$\A_V$ such that the class $\tilde{c}=[\ti L]\in
\Pic(\A_V)$ is cubical, i.e.,  satisfies the relation
$$\ti m_{123}^*\ti c-\ti m_{12}^*\ti c-\ti m_{13}^*\ti c-\ti m_{23}^*\ti c+\ti
m_{1}^*\ti c+\ti m_{2}^*\ti c+\ti m_{3}^*\ti c=0,$$ where
$V\subset U$ is a big open subset and where $\ti m_I$ for
$I\subset \{1,2,3\}$ is the map $\A\times _S\A\times _S\A \to \A$
defined by addition over the factors in $I$.

Now let us choose an ample line bundle $H$ on $S$. Then
the map ${\hat h}_c :A(K)\to \ZZ$ given by
$$ {\hat h}_c (P)= {\deg}_H {(\tilde P -\tilde 0)}^*\tilde{c} $$
is well defined as $\tilde P$ is defined on a big open subset of $S$ and
${\ti P}^*\ti L$ extends to a rank $1$ reflexive sheaf on
$S$. This map is the canonical (N\'eron--Tate) height of $A$ corresponding to $c$
(see \cite[Chapter III, Section 3]{MB}).

\medskip

The following theorem was suggested to the authors by M. Raynaud
(in the good reduction case over a curve $S$,  and with a somewhat
different proof).

\begin{Theorem} \label{thm:height}
Assume that $A$ has potential good reduction.
If $P\in A(K)$ is V-divisible and $c$ is symmetric then
$\hat{h}_c (P)=0$.
\end{Theorem}

\begin{proof}
Let us first assume that $A$ has good reduction. By assumption
there exists a $K$-point $P_n$ of $A^{(n)}$ such that
$V^n(P_n)=P$. { Since $\A\to U$ is an abelian scheme, so is
$\A^{(n)}\to U$, thus $P_n$ is the restriction to $\Spec K$  of
$\tilde P_n\in \A^{(n)}(U)$.   }

Let us factor the absolute Frobenius morphism $F^n_A$ into the
composition of the relative Frobenius morphism $F^n_{A/K}: A\to
A^{(n)}$ and $W_n: A^{(n)}\to A$. Let us set $c_n=W_n^*c$. { Its
cubical extension $\tilde c_n\in \Pic(\A^{(n)}_{V_n})$, for some
big open $V_n\subset U$, together with $H$ allows one to define
$\hat{h}_{c_n} (P_n)$ by the corresponding formula.  } Since
$(F^n_A)^*c=p^nc$, we have $(F^n_{A/K})^*c_n=p^nc$. On the other
hand, since $c$ is symmetric, we have $[p^n]^*c=p^{2n}c$ and hence
$(F^n_{A/K})^* ((V^n)^*c)=p^{2n}c$. Therefore
$$(F^n_{A/K})^* ((V^n)^*c-p^nc_n)=0.$$
Since $F^n_{A/K}$ is an isogeny this implies that the class
$d=(V^n)^*c-p^nc_n$ is torsion. By additivity and functoriality of the canonical height (see
\cite[Theorem, p.~35]{Se}) we have
$$\hat{h}_c (P)= \hat{h}_{(V^n)^*c}(P_n)=
\hat{h}_{p^nc_n}(P_n)+\hat{h}_{d}(P_n) = p^n \cdot \hat{h}_{c_n}
(P_n)$$ (note that additivity implies that
$\hat{h}_{md}=m\hat{h}_{d}$, so since $md=0$ for some $m$, we get
$\hat{h}_{d}=0$).
Therefore if $\hat{h}_c (P)\ne 0$ then $|\hat{h}_c (P)|\ge {p^n}$ and we get a contradiction if $n$ is
sufficiently large.

Now let us consider the general case. Since there { exist} only
finitely many codimension $1$ points $s\in S$ at which $A$ has bad
reduction, one can find a finite Galois extension $K'$ of $K$ such
that if $S'$ is the normalization of $S$ in $K'$ then $A_{K'}$ has
good reduction at every codimension $1$ point $s'\in S'$. { On the
other hand, if $P \in A(K)$ is $V$-divisible on $A$, $P\otimes K'
\in A(K')$ is $V$ divisible on $A_{K'}$.} Then by the above we
have $\hat{h}_{\pi^*c} (P')=0$ and functoriality of the canonical
height implies that $\hat{h}_c (P)=0$.
\end{proof}

\medskip

\begin{Remark}
It is an interesting problem { whether}  {
Theorem~\ref{thm:height} }holds for an arbitrary abelian variety
$A/K$. {  Its} proof shows that one can use the semiabelian
reduction theorem to reduce the general statement to the case when
$A$ has semiabelian reduction (see \cite[7.4, Theorem 1]{BLR}).
\end{Remark}

\medskip

Now assume that $S$ is geometrically connected. Then the extension
$k\subset K$ is regular (i.e.,  $K/k$ is separable and $k$ is
algebraically closed in $K$). Let $(B, \tau)$ be the $K/k$-trace
of the abelian $K$-variety $A$, where $B$ is an abelian
$k$-variety and $\tau : B_K\to A$ is a homomorphism of abelian
$K$-varieties (it exists by \cite[Theorem~6.2]{Co}). Let us recall that by
definition $(B, \tau)$ is a final object in the category of pairs
consisting of an abelian $k$-variety and a $K$-map from the scalar
$K$-extension of this variety to $A$.

Since the extension $k\subset K$ is regular, the kernel $K$-group
scheme of $\tau$ is connected (with connected dual)
(\cite[Theorem~6.12]{Co}). Therefore $\tau$ is injective on
$K$-points and in particular we can treat $B(k)$ as a subgroup of
$A(K).$

\begin{Corollary}\label{V-divisibility}
Assume that $A$ has potential good reduction.
If $P\in A(K)$ is V-divisible then  $[P]\in (A(K)/B(k))_{\tors}$.
In particular, if $k$ is algebraically closed then $P\in
B(k)+A(K)_{\tors}\subset A(K)$.
\end{Corollary}

\begin{proof}
We can choose the class $c\in \Pic (A)$ so that it is ample and
symmetric. Then the first part of the corollary follows from
Theorem~\ref{thm:height}  and \cite[ Theorem~9.15]{Co} (which is
true for regular extensions $K/k$).

To prove the second part take positive integer $m$ such that
$mP=Q\in B(k)$. Since $k$ is algebraically closed, the set $B(k)$
is divisible and there exists $Q'\in B(k)$ such that $mQ'=Q$. Then
$P=Q'+(P-Q')$, where $m(P-Q')=0$.
\end{proof}

\medskip

Let us assume that the field $k$ is algebraically closed. It is an
interesting question { whether}  a $V$-divisible $K$-point $P$ of
$A$ can be written as a sum of $Q+R$, where $Q\in B(k)$ and $R\in
A(K)_{\tors}$ is torsion {\it of order prime-to-}$p$.

By the Lang--N\'eron theorem (\cite[Theorem~2.1]{Co}), the  groups
${A}^{(i)}(K)/ {B}^{(i)}(k)$ are finitely generated. It follows
that the groups $G_i=({A}^{(i)}(K)/ {B}^{(i)}(k))_{\tors}$ are
finite.

Note that the homomorphism $ {B}(k)\to {B}^{(i)}(k)$ induced by
$F^i_{B/k}$ is a bijection.{ One has a factorization $F^i_{A/K}:
{A}(K^{1/{p^i}})\to  A^{(i)}(K) \to {A}^{(i)}(K^{1/p^i})$,
inducing a bijection ${A}(K^{1/{p^i}})\to A^{(i)}(K)$. Thus in
particular,
$$F_i:{A}(K)/{B}(k)\to { A}^{(i)}(K)/{B}^{(i)}(k)$$
is injective. }

Moreover, the Verschiebung morphism induces the homomorphisms
$$V_i:{A} ^{(i)}(K)/{B}^{(i)}(k)\to {A}(K)/{B}(k)$$
such that $V_iF_i=p^i$ and $F_iV_i =p^i$. This shows that
prime-to-$p$ torsion subgroups of groups $G_i$ are isomorphic and
in particular have the same order $m$.

Now let us assume that orders of the $p$-primary torsion subgroups
of the abelian groups $G_i$ are uniformly bounded by some $p^e$.
Then for all $i\ge e$
$$F_i(m[P])=F_i(V_i(m[P_i]))=p^im[P_i]=0 .$$
This implies that $m[P]=0$, so $mP\in B(k)$. Now $B(k)$ is a
divisible group so there exists some $Q'\in B(k)$ such that
$mP=mQ$. Then $R=P-Q\in A(K)$ is torsion of order prime to $p$. So
we conclude

\begin{Lemma}
If the order of the $G_i$ is bounded as $i$ goes to infinity, under the
assumption the Theorem~\ref{thm:height}, there exists a positive
integer $m$, prime to $p$ and such that  $m\cdot P_i \in B(k)$ for every integer
$i$.
\end{Lemma}

Note that the above assumption on $G_i$ is satisfied, e.g., if $A$ is an
elliptic curve over the function field $K$ of a smooth curve over
$k=\bar k$. If $A$ is isotrivial then the assertion is clear. If
$A$ is not isotrivial then the $j$-invariant of $A$ is
transcendental over $k$. In this case $A(K^{\rm perf})_{\rm tors}$ is
finite (see \cite{Le}) so orders of the groups $G_i=A^{(i)}(K)_{\tors}$
are uniformly bounded.

\section{Stratified bundles}

In this section we use the height estimate of the previous section
and the fact that torsion stratified line bundles on a perfect
field have order prime to $p$ (apply Proposition~\ref{prop:BD}
together with Lemma~\ref{lem:finiteness}).

Let $k$ be an algebraically closed field of positive
characteristic $p$. Let $f: X\to S$ be a smooth projective
morphism of $k$-varieties with geometrically connected fibres.
Assume that $S$ is projective, { which surely is a very strong
assumption. Indeed, if $k \neq \FF_p$ , and in the statement of
Theorem~\ref{stratified-line-bundles}, $S'$ is open, then one
obtains the stronger Theorem~\ref{stratified-vb}. } For
simplicity, let us also assume that $f$ has a section $\sigma:
S\to X$. Consider the torsion component $\Pic ^{\tau} (X/S)\to S$
of identity of $\Pic (X/S)\to S$. Let $\varphi _n: \Pic (X/S)\to
\Pic (X/S)$ be the multiplication by $n$ map. Then there exists an
open subgroup scheme $\Pic ^{\tau} (X/S)$ of $\Pic (X/S)$ such
that every geometric point $s$ of $S$ the fibre of $\Pic ^{\tau}
(X/S)$ over $s$ is the union
$$\bigcup _{n>0}\varphi _n ^{-1} (\Pic ^0(X_s)),$$
where $\Pic ^0(X_s)$ is the connected component of the identity of
$\Pic (X_s/s)$. It is well known that $\Pic ^{\tau} (X/S)\to S$ is
also a closed subgroup scheme of $\Pic (X/S)$. Moreover, the
morphism $\Pic ^{\tau} (X/S)\to S$ is projective and the formation
of $\Pic ^{\tau} (X/S)\to S$ commutes with a base change of $S$
(see, e.g., \cite[Theorem 6.16~and~Exercise 6.18]{Kl}).

We assume that $\Pic ^{0}(X_{s})$ is reduced for every point $s\in
S$.

\begin{Theorem}\label{stratified-line-bundles}
Let $\LL=\{L_i, \sigma _i\}$ be a relatively stratified line
bundle on $X/S$. Assume that there exists a dense subset
$S'\subset S(k)$ such that for every $s\in S'$ the stratified
bundle $\LL _s=\LL |_{X_s}$ has finite monodromy. Then $\LL _{\bar
\eta}$ has finite monodromy.
\end{Theorem}

\begin{proof}
Replacing $\LL$ by a power $\LL^{\otimes N}$, where $N$ is sufficiently large,
we may assume that $\LL_s\in \Pic^0(X_s)$ for all closed points $s$ in $S$
(see \cite[Corollary 6.17]{Kl}).

By assumption ${\hat \pi}: \hat \A=\Pic ^{0} (X/S)\to S$ is an
abelian scheme. Let us consider the dual abelian scheme $\A \to
S$. We have a well defined Albanese morphism $g: (X, \sigma )\to
(\A, e)$ (see \cite[Expos\'e~VI,~Th\'eor\`eme~3.3]{FGA}).
Moreover, the map $g^*: \Pic ^0(\A/S)\to \hat \A=\Pic ^0 (X/S)$ is
an isomorphism of $S$-schemes. Let us set  ${\hat A}={\hat
\A}_{\eta}$.

Let $P_i$ be the $K$-point of ${\hat A}^{(i)}$ corresponding to
$(L_i)_{\eta}$. Note that the $K$-point $P_0\in \hat A$ is
$V$-divisible. Indeed, by the definition of a relative
stratification we have $V^n(P_n)=P_0$ for all integers $n$.
Similarly, we see that all the points $P_i\in {\hat A}^{(i)}(K)$
are $V$-divisible. By Corollary \ref{V-divisibility} it follows
that $P_i\in {\hat B}^{(i)}(k)+{\hat A}^{(i)}(K)_{\tors}$, where
$({\hat B}/k, \hat \tau :{\hat B}_K\to \hat A)$ is the $K/k$-trace
of $\hat A$ (note that $({\hat B}^{(i)}/k, {\hat \tau} ^{(i)})$ is
the $K/k$-trace of ${\hat A}^{(i)}$). So for every $i\ge 0$ we can
write $P_i=Q_i+R_i$ for some $Q_i\in  {\hat B}^{(i)}(k)$ and
$R_i\in {\hat A}^{(i)}(K)_{\tors}$.

Now we transpose the above by duality. Let $A$ be the dual abelian
$K$-variety of $\hat A$ and $B$ the dual abelian $k$-variety of
$\hat B$. We have the $K/k$-images $\tau ^{(i)}: A_{\eta}^{(i)}\to
B^{(i)}_K$ and an $S$-morphism $\tau: \A \to B\times _k S$
(possibly after shrinking $S$). By abuse of notation we can treat
$L_i$ as line bundles on $\A$ because $g^*: \Pic ^0(\A/S)\to \Pic
^0 (X/S)$ is an isomorphism. Let $M_i$ be the line bundle on
$B^{(i)}$  corresponding to $Q_i$ and let $\pi _i: B^{(i)}\times
_k S\to B^{(i)}$ denote the projection. Let us fix a non-negative
integer $i$ and take a positive integer $n_i$ such that $n_i
R_i=0$. Then the line bundle $L_i^{\otimes n_i }\otimes
\tau^*\pi_i ^*M_i^{\otimes -n_i}$ has degree $0$ on every fiber of
$\A\to S$. Thus  it is trivial after restriction to $\A_{\eta}$.
Hence after shrinking $S$ we can assume that $L_i^{\otimes n_i
}\simeq \tau^*\pi_i ^*M_i^{\otimes n_i}$.

Let us fix a point $s\in S(k)$ and consider the morphism  $$
\pi_i'=(\tau^{(i)} \pi _i)_{\A^{(i)}_s}:\A^{(i)}_s\to B^{(i)}.$$
 Note that $\tau ^{(i)}$ has
connected fibres and hence  $(\pi
_i')_*\O_{\A^{(i)}_s}=\O_{B^{(i)}}$. By assumption there exists a
positive integer $a_s$, such that for every $i$ the order of the
line bundle $(L_i)_{\A_s}$ divides $a_s$. The important point is
that $a_s$ is prime to $p$.

Therefore $(\pi_i')^*M_i^{\otimes a_sn_i}\simeq \O_{A_s}$ and by
the projection formula
$$M_i^{\otimes a_sn_i}\simeq(\pi_i')_*(\pi_i')^*M_i^{\otimes
a_sn_i}\simeq (\pi_i')_*\O_{A_s}\simeq \O_{B}.$$ This implies that
$M_i$ is a torsion line bundle and hence $Q_i\in {\hat
A}^{(i)}(K)_{\tors}$. Therefore $$P_i=Q_i+R_i\in {\hat
A}^{(i)}(K)_{\tors}.$$

Let us recall that the set of $p$-torsion points of ${\hat A}(K)$
is finite. Assuming it is not empty, we can therefore find a
non-empty open subset $U\subset S$ such that for every $s\in U(k)$
and every $p$-torsion point $T\in {\hat A}(K)$ the section $\ti T$
is defined on $U$ and the point ${\ti T}(s)$ is non-zero.

Let us write the order of $P_i$ as $m_ip^{e_i}$, where $m_i$ is
not divisible by $p$. If $e_0\ge 1$ then the point
$m_0p^{e_0-1}P_0$ is $p$-torsion  in ${\hat A}(K)$. If we take
$s\in S'\cap U(k)$,  then $a_sm_0p^{e_0-1}{\ti
P}_0(s)=[L_0^{\otimes a_sm_0}]_s=0$, a contradiction. It follows
that $m_0P_0=0$. Similarly, the order of all $P_i$ is prime to
$p$.

As already mentioned in the last section,  the homomorphism ${\hat
A} (K^{1/{p^i}})\to {\hat A}^{(i)}(K)$ induced by $F^i_{A/K}$ is a
bijection. So we have an induced injection
$$F_i:{\hat A}(K)\to {\hat A}^{(i)}(K).$$
On the other hand, the Verschiebung morphism induces homomorphisms
$$V_i:{\hat A} ^{(i)}(K)\to {\hat A}(K)$$
such that $V_iF_i(P)=p^iP$ and  $F_iV_i (Q)=p^i Q$   for all $P\in
\hat A(K)$ and $Q\in {\hat A}^{(i)}(K)$. Hence
$$p^im_0P_i=F_iV_i(m_0P_i)=F_i(m_0P_0)=0$$
and since the order of $P_i$ is prime to $p$ we have $m_0P_i=0$
for all $i\ge 0$. Therefore $(L_i)_{\bar \eta}^{\otimes m_0}\simeq
\O_{X_{\bar\eta}}$ for all $i$ and the stratified line bundle $\LL
_{\bar \eta}$ has finite monodromy.
\end{proof}

\medskip

Now we fix the following notation: $k$ is an algebraically closed
field of positive characteristic $p$ and $f: X\to S$ is a smooth
projective morphism of $k$-varieties with geometrically connected
fibres.

\begin{Theorem}\label{stratified-vb}
Let $\EE=\{E_i, \sigma _i\}$ be a relatively stratified bundle on
$X/S$. Assume that there exists a dense subset $U\subset S(k)$
such that for every $s\in U$ the stratified bundle $\EE _s=\EE
|_{X_s}$ has finite monodromy of order prime to $p$.
\begin{itemize}
 \item[1)] Then there
exists a finite Galois \'etale covering $\pi _{\bar \eta}: Y_{\bar
\eta}\to X_{\bar \eta}$  of order prime-to-$p$ such that
$\pi _{\bar \eta}^*\EE_{\bar \eta}$ is a direct sum of stratified
line bundles.
\item[2)]  If {$k\neq \bar \FF_p$ } and $U$ is open in $S(k)$, then the monodromy
group
of $\EE_{\bar \eta}$ is finite,  and $\EE_{\bar \eta}$ trivializes on a finite
\'etale
cover $Z_{\bar \eta}\to X_{\bar \eta}$ which factors as a Kummer
(thus finite abelian of order prime to $p$) cover $Z_{\bar
\eta}\to Y_{\bar \eta}$ and a Galois cover $Y_{\bar \eta}\to
X_{\bar \eta}$ of order prime to $p$.
\end{itemize}

\end{Theorem}

\begin{proof}
We prove 1). Let us first remark that the schemes $X^{(i)}_{\bar
\eta}$, $i\ge 0$, are all isomorphic (as schemes, not as
$k$-schemes). Therefore the relative Frobenius induces an
isomorphism on fundamental groups.

By the first part of  Theorem~\ref{GK-etale}  we know that there
exists a finite Galois \'etale covering $\pi _i: Y_{\bar \eta
,i}\to X^{(i)}_{\bar \eta}$ of degree prime to $p$ such that
$\pi_i^*(E _i)$ is a direct sum of line bundles $\oplus_1^r
L_{ij}$. Note that from the proof of Theorem~\ref{GK-etale} the
degree of $\pi _i$ depends only on $\pi _1^{p'}(X_{\bar
\eta}^{(i)}, \sigma ^{(i)}(\bar \eta))$ { and the Brauer-Feit
constant $j(r)$}, and therefore it can be bounded independently of
$i$. Using the Lang--Serre theorem (see \cite[{Th\'eor\`eme}
4]{LS}) we can therefore assume that $Y_{\bar \eta, i}=Y_{\bar
\eta}^{(i)}$, where $Y_{\bar \eta}=Y_{\bar\eta, 0}$. Now we know
that
$$\oplus_{j=1}^r L_{ij}\simeq
(F_{Y^{(i)}_{\bar\eta}/\bar\eta}^{i})^*\big(\oplus_{j'=1}^r
L_{i+1,j'}\big).$$ { By the Krull-Schmidt theorem, the set of
isomorphism classes of line bundles $\{L_{ij}\}_j $ is  the same
as the set of isomorphism classes of lines bundles which come by pull-back
$\{
(F_{Y^{(i)}_{\bar\eta}/\bar\eta}^{i})^*(L_{i+1,j'})  \}_{j'} .$
So we can reorder the indices $j'$  so that
$$(F_{Y^{(i)}_{\bar\eta}/\bar\eta}^{i})^*(L_{i+1,j}) \cong L_{i,j}.$$
This finishes the proof of 1).

To prove 2), we do the proof 1) replacing $Y_{\bar \eta}\to X_{\bar \eta}$ by
$Z_{\bar \eta}\to X_{\bar \eta}$ of Theorem~\ref{GK-etale} 2). This finishes
the proof of 2).

 }
\end{proof}

\medskip

\begin{Remarks}
\begin{itemize}
\item[1)] Case 2) of Theorem~\ref{stratified-vb} applied to a line bundle
extends Theorem~\ref{stratified-line-bundles}, where $S$ was
assumed to be projective, { $\Pic^0(X_s)$ reduced for all $s\in S$
closed,} $S'\subset S(k)$ dense, to the case when $S$ is not
necessarily projective and $S'\subset S(k)$ is open and dense, but
we have to assume that $k$ is uncountable.
\item[2)] If $Y_{\bar\eta}$ has a good projective model satisfying
assumptions of Theorem \ref{stratified-line-bundles} then it
follows that $\EE_{\bar \eta}$ has finite monodromy.
\end{itemize}
\end{Remarks}

\end{document}